\documentclass[reqno,10pt]{amsart}
\usepackage{amsmath,amssymb,amsthm,mathtools}
\usepackage[margin=1.05in]{geometry}
\usepackage{hyperref}
\hypersetup{colorlinks=true,linkcolor=blue,citecolor=black,urlcolor=blue}

\newtheorem{counter}{}[section]
\theoremstyle{definition}
\newtheorem{definition}[counter]{Definition}

\newtheorem{example}[counter]{Example}
\newtheorem{remark}[counter]{Remark}

\theoremstyle{plain}
\newtheorem{lemma}[counter]{Lemma}
\newtheorem{proposition}[counter]{Proposition}
\newtheorem{theorem}[counter]{Theorem}
\newtheorem{corollary}[counter]{Corollary}

\newcommand{\C}{\mathcal{C}}
\newcommand{\D}{\mathcal{D}}
\newcommand{\B}{\mathcal{B}}
\newcommand{\N}{\mathcal{N}}
\newcommand{\T}{\mathcal{T}}
\newcommand{\Z}{\mathcal{Z}}

\newcommand{\kk}{\Bbbk}
\newcommand{\Vect}{\mathsf{Vect}}
\newcommand{\sVect}{\mathsf{sVect}}
\newcommand{\Rep}{\mathrm{Rep}}
\newcommand{\Hom}{\mathrm{Hom}}
\newcommand{\Inv}{\mathrm{Inv}}
\newcommand{\unit}{\mathbf{1}}
\newcommand{\loc}{\mathrm{loc}}
\newcommand{\id}{\mathrm{id}}
\newcommand{\FPdim}{\mathrm{FPdim}}
\newcommand{\fp}{\mathfrak{p}}
\newcommand{\fu}{\mathfrak{u}}

\newcommand{\Wns}{\mathcal{W}^{\mathrm{ns}}}
\newcommand{\bA}{\bar{A}}
\newcommand{\sloc}{\text{s-loc}}
\newcommand{\mon}{\mathrm{mon}}
\newcommand{\gr}{\mathrm{gr}}
\newcommand{\ev}{\mathrm{ev}}
\newcommand{\coev}{\mathrm{coev}}

\begin{document}

\title{Commutative exact superalgebras and super-modular tensor categories}

\author{A. Zuevsky}

\address{Institute of Mathematics, Czech Academy of Sciences, \v{Z}itn\'{a} 25, Prague, Czech
		Republic}
\email{zuevsky@yahoo.com}

\begin{abstract}
We consider the theory of commutative exact algebras in braided finite tensor categories 
 to superalgebras, motivated by extensions of vertex operator (super)algebras.
For a braided category $\C$ let $\B=\C^{\gr}$ be the category of $\mathbb Z/2$-graded objects 
with the Koszul-signed braiding. Its even-morphism part is the underlying category of the 
supercategory of Creutzig-Kanade-McRae. A supercommutative superalgebra $A=A_0\oplus A_1$ in $\C$ 
is the same as a commutative algebra $\bA$ in $\B$, and super-local modules are the local 
$\bA$-modules. Hence, rigidity, pivotality and ribbon statements are instances of the ordinary 
theorems of Shimizu-Yadav, with the ribbon condition $\theta_{A_0}=\id$, 
$\theta_{A_1}=\varepsilon\,\id$, where $\varepsilon=\pm1$ is fixed by the statistics of the odd part. 
For $\C$ non-degenerate and $\bA$ haploid and exact we prove that the category of super-local 
modules is a braided finite tensor category whose M\"uger center is $\mathsf{sVect}$. It is the 
centralizer of a fermion in a non-degenerate category obtained by condensing in 
$\C\boxtimes(\text{toric code})$.  If $\C$ is ribbon and $\bA$ is symmetric 
Frobenius it is ribbon. We give a self-contained treatment of graded simple current algebras and of 
the free fermion (Ising) example.  We compare our formulation with 
Creutzig-Kanade-McRae and Creutzig-McRae-Shimizu-Yadav. 
\end{abstract}



\maketitle

\section{Introduction}

The paper \cite{SY26} determines when the category $\C_A^{\loc}$ of local modules over 
a commutative exact algebra $A$ in a braided finite tensor category $\C$ is rigid, pivotal, 
ribbon, or a (possibly non-semisimple) modular tensor category. The motivating example 
is a vertex operator algebra (VOA) extension $U\subset W$, with $\C=\Rep(U)$, $A=W$ and 
$\C_A^{\loc}\simeq\Rep(W)$ \cite{HKL15, CKM24}. For vertex operator superalgebras (VOSAs) 
a VOSA extension of $U$ is a supercommutative superalgebra in a category of $U$-modules, 
and the category of its (untwisted) modules is the category of local modules 
\cite{CKL20, CKM24}. The semisimple theory is given in \cite{KO02, CKM24}. 
Rigidity for non-semisimple categories is the subject of \cite{CMSY25}. According to its 
introduction, Section 4.3 there treats VOSA extensions. 

 Let us describe our approach. 
Following \cite{CKM24}, a superalgebra in $\C$ is an algebra in the category of graded 
objects.  The Koszul signs are placed in the braiding of that category 
(our $c^{\mathrm s}$, Section \ref{secgraded}). With this convention a superalgebra 
is an ordinary commutative algebra $\bA$ in an ordinary braided category $\B$, 
 thus \cite{SY26} applies to $\bA$. But modularity needs extra care. 
 $\B$ contains the transparent fermion $\unit\boxtimes f$, 
thus non-degeneracy is replaced by the condition that the M\"uger center is  $\mathsf{sVect}$. 
We prove this by embedding $\B$ in a modular category constructed with the toric code.

\subsection{Main results of this paper}

All notation is introduced in Section \ref{secfound}. $\B=\C^{\gr}$ is the graded category, 
$\bA=(A_0,A_1)$ the algebra in $\B$ attached to a superalgebra $A$, $\varepsilon\in\{\pm1\}$ 
the twist of the fermion, and $\C_A^{\sloc}:=\B_{\bA}^{\loc}$. For an algebra $\bA$ in $\B$ we write 
${}_{\bA}\B_{\bA}$ for the category of $\bA$-bimodules in $\B$, and $\mathcal E'$ for the M\"uger 
center of a braided category $\mathcal E$, as in \cite{SY26}.

\begin{theorem} 
\label{thmAs}
Let $\C$ be an additive rigid braided category with equalizers and coequalizers, 
and $A$ a supercommutative superalgebra in $\C$ (Definition \ref{defsuperalgebra}). 
\begin{enumerate}

\item[(a)] If the category ${}_{\bA}\B_{\bA}$ of $\bA$-bimodules in $\B$ is rigid, then $\C_A^{\sloc}$ is rigid.

\item[(b)] If $\C$ is pivotal, ${}_{\bA}\B_{\bA}$ is rigid, and $\fu_{A_0}=\fp_{A_0}$, 
$\fu_{A_1}=\varepsilon\,\fp_{A_1}$ (Drinfeld isomorphism $\fu$ and pivotal structure 
$\fp$ of $\C$), then $\C_A^{\sloc}$ is pivotal.

\item[(c)] If $\C$ is ribbon, ${}_{\bA}\B_{\bA}$ is rigid, $\theta_{A_0}=\id$, 
$\theta_{A_1}=\varepsilon\,\id$ and $\bA^*\in(\C_A^{\sloc})'$, then $\C_A^{\sloc}$ is ribbon.
\end{enumerate}
\end{theorem}

\begin{remark}
\label{remAstar}
In (c), $(\C_A^{\sloc})'$ is the M\"uger center,  
 thus the condition says 
that the dual $\bA^*$, equipped with the right $\bA$-module structure of \cite{SY26}, is 
transparent in $\C_A^{\sloc}$. If $\theta^{\mathrm s}_{\bA}=\id$ (which is the assumption 
$\theta_{A_0}=\id$, $\theta_{A_1}=\varepsilon\,\id$, Proposition \ref{propbijection} (c)) 
this module is 
automatically local, and $\theta^{\mathrm s}$ is a twist of $\C_A^{\sloc}$ 
\cite{SY26}. The extra condition is then equivalent to $\theta^{\mathrm s}$ 
being a ribbon structure, i.e., compatible with duality. It holds, for 
instance, if $\bA\cong\bA^*$ as right $\bA$-modules 
(e.g., $\bA$ Frobenius, since the Frobenius isomorphism is one of right modules), since then 
$\bA^*\cong\bA$ is the unit of $\C_A^{\sloc}$. In Theorem \ref{thmBs} (2) below 
the condition is not assumed. Ribbon property of $\C_A^{\sloc}$ is obtained from an 
ambient modular category.
\end{remark}

\begin{definition}
\label{defsupernd}

A braided finite tensor category $\D$ with a distinguished invertible object $f$, 
$f\otimes f\cong\unit$, $c_{f,f}=-\id$, is \emph{super-non-degenerate} if its M\"uger center 
$\D'$ equals $\langle\unit,f\rangle\simeq\mathsf{sVect}$ (a property of the braided category 
with the object $f$). If $\D$ is moreover ribbon with $\theta_f=\varepsilon\,\id$, 
it is called \emph{$\varepsilon$-super-modular}. This includes the choice of the ribbon 
structure.  
\end{definition}

\begin{remark}
\label{remsphericity}
Since $\theta_f=q(f)\dim f=-\dim f$ (Remark \ref{remtheta-vs-c}), 
$\dim f=-\varepsilon$. Thus, $(-1)$-super-modular categories have $\dim f=1$ and, 
when spherical (and pseudo-unitary), are the super-modular categories of \cite{BGHNPRW17}. 
$(+1)$-super-modular categories have $\dim f=-1$, so they are not pseudo-unitary 
(not all dimensions are positive). 
  The sign of $\dim f$ is not important on sphericality: for any 
pivotal structure one has $\dim_r X=\dim_l X^*$, \cite{EGNO16}. Hence, every pivotal 
structure on a category all of whose simple objects are self-dual 
(e.g.,  $\sVect^\varepsilon$ 
and the toric code $\T^\varepsilon$ of Section \ref{sectoric}, for both signs of $\varepsilon$) 
is spherical.
\end{remark}

\begin{theorem}
\label{thmBs}

Let $\C$ be a braided finite tensor category over an algebraically closed field of 
characteristic $0$ and $A$ a haploid supercommutative superalgebra with $\bA$ exact. Then: 
\begin{enumerate}

\item[(0)] $\C_A^{\sloc}$ is a braided finite tensor category. Moreover, 
${}_{\bA}\B_{\bA}$ is rigid, 
so that the rigidity assumption in Theorem \ref{thmAs} (a)-(c) is automatic, while the other 
assumptions of Theorem \ref{thmAs} (b), (c) (on $\fu$, $\fp$, $\theta$, $\bA^*$, and the 
assumption that $\C$ be pivotal, resp.\ ribbon) are still needed for pivotality, resp. 
ribbonness of $\C_A^{\sloc}$.

\item[(1)] If $\C$ is non-degenerate (braided), then 
$\C_A^{\sloc}$ is super-non-degenerate with fermion $\tilde f:=\Pi\bA=(A_1,A_0)$.
 More precisely, there are a non-degenerate braided finite tensor category $\N$ and
 an invertible $\tilde f\in\N$ with 
$\C_A^{\sloc}=\mathcal Z_{(2)}(\langle\tilde f\rangle\subset\N)$ and 
$\FPdim\N=2\,\FPdim\C_A^{\sloc}$.

\item[(2)] If in addition $\C$ is ribbon (hence an MTC) and $\bA$ is symmetric Frobenius, 
i.e., Frobenius with $\theta_{A_0}=\id$, $\theta_{A_1}=\varepsilon\id$ (Definition \ref{defprops}), 
then $\N$ is a modular tensor category and $\C_A^{\sloc}$ is $\varepsilon$-super-modular.
\end{enumerate}
\end{theorem}

\begin{corollary}
\label{corext}
Under (2), $\C_A^{\sloc}$ has a modular extension $\N$ of twice its 
Frobenius-Perron dimension.
\end{corollary}

We provide also graded simple currents (Proposition  \ref{propsimple-current}), examples 
(Section \ref{secexamples}), the VOSA notions (Section \ref{secVOSA}),
 Witt classes (Section \ref{secWitt}), a comparison with \cite{CKM24, CMSY25} 
(Section \ref{seccomparison}). 

Theorems \ref{thmAs} and \ref{thmBs} (0) require the material of \cite{SY26, CMSY25} 
 applied to $(\B,\bA)$. Proposition \ref{propbijection} and Lemma 
\ref{lemdualityB} use the Koszul sign rule. Lemma \ref{leminduction} reflects 
 induction/monodromy statements \cite{Par95, KO02}. The new content of the paper is: 
the embedding $\B\hookrightarrow\C\boxtimes\T^\varepsilon$ (Lemma \ref{lemtoric}) together with the 
proof that the image of an exact algebra is exact (Lemma \ref{lemexact}); 
the resulting description 
of $\C_A^{\sloc}$ as a centralizer, with its M\"uger center and its modular extension 
(Theorem \ref{thmBs} (1), (2)); the self-contained proof of Proposition \ref{propsimple-current}; and the 
verification of signs, twists and dimensions in the examples. We do not treat the Ramond sector, other 
modular completions of $\sVect$, or a super Witt group (Remark \ref{remenriched}, Section \ref{secWitt}).

\section{Foundations}
\label{secfound}

In this paper, finite tensor categories (FTCs) are over an algebraically closed field $\kk$ of 
characteristic $0$. This assumption is used in Theorem \ref{thmBs}, Lemma \ref{lemexact}, 
Proposition \ref{propsimple-current} and Section \ref{secVOSA}: algebraic closedness for the 
results of \cite{Shi19, SY26} on finite tensor categories, characteristic $0$ for the separability of 
group algebras and for the existence of the toric code. Theorem \ref{thmAs} and Sections 
\ref{secgraded}-\ref{secrigid} hold over an arbitrary field, since they use  
\cite{SY26}, 
which is formulated over an arbitrary field.  
We follow \cite{SY26} for the notions of closed, rigid, pivotal, ribbon categories 
($\fp=\fu\theta$), the Drinfeld isomorphism $\fu$, central algebras, $\C_A$, ${}_A\C_A$, 
$\C_A^\sigma$, Frobenius and symmetric Frobenius algebras, exact algebras, haploid algebras, 
M\"uger center $\C'$, centralizers 
$\mathcal Z_{(2)}(\D\subset\C)=\{X\mid c_{Y,X}c_{X,Y}=\id\ \forall Y\in\D\}$, 
non-degenerate braided FTCs (those with $\C'\simeq\Vect$), and MTCs (i.e., non-degenerate ribbon FTCs). 
We write $\mon_{X,Y}=c_{Y,X}\circ c_{X,Y}$. We use right modules. For local modules 
the left and right conventions agree (Sections \ref{secinduction} and \ref{seccomparison}). 
A tensor full subcategory of an FTC is a full subcategory closed under subquotients, tensor products 
and duals. In this paper, $\C_A^{\sloc}$ denotes the category of super-local modules, 
$\B_{\bA}^{\loc}$ the category of local modules in the ordinary braided category $\B$ (these are the 
same category, Proposition \ref{propbijection}), and $\D^{\mathrm s}$ (see Remark \ref{remenriched}) 
the enriched supercategory constructed on $\D$. 

\subsection{Quadratic data of invertible objects}

For an invertible object $X$ of a braided category, $X\otimes X$ is simple (if $\unit$ is) 
and $c_{X,X}=q(X)\,\id_{X\otimes X}$ for a scalar $q(X)$.

\begin{remark}
\label{remtheta-vs-c}
In a ribbon category, $\theta_X=\mathrm{ptr}(c_{X,X})$, the partial trace taken with the 
pivotal structure $\fp=\fu\theta$ (the convention of \cite{SY26}). Hence 
$\theta_X=q(X)\dim X$ for invertible $X$. With the opposite convention $\theta^{-1}$ 
in place of $\theta$ nothing changes below, since only $\theta=\pm1$ occurs. 
Thus, $c_{X,X}$ and $\theta_X$ are independent data unless $\dim X$ is fixed. 
Balancing $\theta_{X\otimes Y}=(\theta_X\otimes\theta_Y)\mon_{X,Y}$ gives only 
$\theta_{X\otimes X}=\theta_X^2q(X)^2$.
\end{remark}

\subsection{The fermion \texorpdfstring{$\mathsf{sVect}^\varepsilon$}{sVect-epsilon}}

$\sVect$ is the category of finite-dimensional $\mathbb Z/2$-graded vector spaces 
with braiding $c(v\otimes w)=(-1)^{|v||w|}w\otimes v$; $f$ is the odd line, 
$f\otimes f\cong\unit$, $c_{f,f}=-\id$, $f$ transparent. For $\varepsilon\in\{\pm1\}$ 
let $\sVect^\varepsilon$ be $\sVect$ with the ribbon structure $\theta_\unit=\id$, 
$\theta_f=\varepsilon\id$ (balancing holds since $\mon_{f,f}=\id$). By Remark 
\ref{remtheta-vs-c}, $\dim f=-\varepsilon$.

\subsection{The graded category}
\label{secgraded}

Let $\C$ be an additive braided monoidal category (finite direct sums exist and 
$\otimes$ is additive in each variable; if $\C$ is rigid the latter is automatic 
since $X\otimes-$ has adjoints). Define $\B=\C^{\gr}$ by the following:

\textit{1.} \; objects are pairs $X=(X_0,X_1)$ of objects of $\C$; $|x|\in\mathbb Z/2$ denotes 
the degree of a component;

 \textit{2.} \; morphisms $X\to Y$ are pairs $(f_0,f_1)$ of morphisms of $\C$;

 \textit{3.} \; $(X\otimes Y)_k=\bigoplus_{i+j=k}X_i\otimes Y_j$, unit $(\unit,0)$, associator 
and unitors induced from $\C$ (they involve no signs);

 \textit{4.} \; $c^{\mathrm s}_{X,Y}=(-1)^{ij}c_{X_i,Y_j}$ on $X_i\otimes Y_j\to Y_j\otimes X_i$;

 \textit{5.}\; if $\C$ has a twist $\theta$ and $\varepsilon\in\{\pm1\}$ is fixed: 
$\theta^{\mathrm s}=\varepsilon^{i}\theta_{X_i}$ on the component $X_i$;

 \textit{6.}\; if $\C$ is rigid, $X^*=(X_0^*,X_1^*)$ with $\ev_X=\ev_{X_0}\oplus\ev_{X_1}$ and 
$\coev_X=\coev_{X_0}\oplus\coev_{X_1}$ (with no signs; these are morphisms of degree $0$).

Then $\C^{\gr}$ is $\C\boxtimes\sVect$ whenever the Deligne product exists (this is automatic 
for finite abelian categories). It has equalizers 
and coequalizers if $\C$ does (for each component), and is rigid if $\C$ is (Lemma \ref{lemzigzag}). 
The hexagons for $c^{\mathrm s}$ follow from those of $c$ since $(-1)^{ij}$ is bilinear in the degrees; 
$\theta^{\mathrm s}$ satisfies balancing since 
$\varepsilon^{i+j}=\varepsilon^i\varepsilon^j$ and Lemma \ref{lemmonodromy}. 
Thus, $\B$ is braided (ribbon if $\C$ is), and $(0,\unit)$ is a transparent invertible 
object $f$ with $f\otimes f\cong\unit$, $c^{\mathrm s}_{f,f}=-\id$, 
$\theta^{\mathrm s}_f=\varepsilon$. If $\C$ is an FTC, so is $\B$ (its unit $(\unit,0)$ has endomorphism 
algebra $\kk$). 

\begin{lemma}
\label{lemzigzag}
Let $\C$ be rigid. Then $\ev_X$ and $\coev_X$ of item $6.$ satisfy the zig-zag identities 
in $\B$, thus $X^*$ is a left dual of $X$ in $\B$. Right duals are obtained in the same way. 
\end{lemma}

\begin{proof}
The monoidal structure of $\B$ (items 1.-3. above) is that of 
$\C\boxtimes\Vect_{\mathbb Z/2}$, thus it involves no signs. 
 Signs enter only through $c^{\mathrm s}$.  
 Now $(X^*\otimes X)_0=X_0^*\otimes X_0\oplus X_1^*\otimes X_1$, so $\ev_X$ is well defined of degree $0$ 
(and zero on the mixed summands $X_i^*\otimes X_j$, $i\neq j$), and likewise 
$\coev_X:\unit\to X_0\otimes X_0^*\oplus X_1\otimes X_1^*$. The zig-zag composite 
$(\id_X\otimes\ev_X)(\coev_X\otimes\id_X)$, restricted to the summand $X_i$, passes only through the 
summands $X_i\otimes X_i^*\otimes X_i$ and is therefore the zig-zag of $X_i$ in $\C$, i.e., $\id_{X_i}$. 
The other identity is analogous.
\end{proof}

\begin{lemma}
\label{lemmonodromy}
$\mon^{\mathrm s}_{X,Y}=\mon_{X,Y}$ componentwise.
\end{lemma}

\begin{proof}
$\mon^{\mathrm s}$ is the composite of 
$c^{\mathrm s}_{X_i,Y_j}$ and $c^{\mathrm s}_{Y_j,X_i}$. The two signs $(-1)^{ij}$ cancel.
\end{proof}

The following lemma concerns the duality data of $\B$. 

\begin{lemma}
\label{lemdualityB}

Let $\C$ be rigid braided. Then:

 (i) The Drinfeld isomorphism of $\B$ is 
$\fu^{\mathrm s}=(-1)^{i}\fu_{X_i}$ on $X_i$.

(ii) If $\fp$ is a pivotal structure of $\C$, 
then $\fp^\varepsilon:=(-\varepsilon)^{i}\fp_{X_i}$ is a pivotal structure of $\B$ 
($\fp^\varepsilon=\fp$ for $\varepsilon=-1$).

 (iii) If $\C$ is ribbon and $\fp=\fu\theta$, 
then $\fp^\varepsilon=\fu^{\mathrm s}\theta^{\mathrm s}$. (iv) $\dim_{\B}f=-\varepsilon$.
\end{lemma}

\begin{proof}
(i) With the conventions $\ev_X:X^*\otimes X\to\unit$, $\coev_X:\unit\to X\otimes X^*$,
\[
\fu_X=(\ev_X\otimes\id_{X^{**}})\circ(c_{X,X^*}
\otimes\id_{X^{**}})\circ(\id_X\otimes\coev_{X^*}):X\to X^{**}.
\]
In $\B$ the evaluations and coevaluations are those of $\C$ componentwise (Lemma \ref{lemzigzag}), 
and $\fu^{\mathrm s}_X$ involves  one braiding $c^{\mathrm s}_{X,X^*}$. On the degree-$i$ 
component, $X_i$ and $X_i^*$ both have degree $i$, thus 
$c^{\mathrm s}_{X_i,X_i^*}=(-1)^{i\cdot i}c_{X_i,X_i^*}=(-1)^ic_{X_i,X_i^*}$.

(ii) A pivotal structure is a monoidal  isomorphism $\id\Rightarrow(-)^{**}$. 
Naturality for degree-preserving morphisms is clear. The canonical isomorphisms 
$(X\otimes Y)^{**}\cong X^{**}\otimes Y^{**}$ are constructed from $\ev,\coev$ and associators, 
with no braiding and hence no signs. 
Thus, monoidality means 
$\fp^\varepsilon_{X\otimes Y}=\fp^\varepsilon_X\otimes\fp^\varepsilon_Y$ on $X_i\otimes Y_j$, 
i.e., $(-\varepsilon)^{i+j}=(-\varepsilon)^i(-\varepsilon)^j$,  since 
$(-\varepsilon)^2=1$.
(iii) $\fu^{\mathrm s}\theta^{\mathrm s}=(-1)^i\varepsilon^i\fu\theta=(-\varepsilon)^i\fp$.
(iv) $\dim_{\B}f=\ev_{f^*}\circ(\fp^\varepsilon_f\otimes\id_{f^*})\circ\coev_{f}$ equals 
$\fp^\varepsilon_f=-\varepsilon$ times $\dim_\C\unit=1$.
\end{proof}

The pivotal structure $\fp^\varepsilon$ of $\B$ is not the componentwise pivotal structure 
of $\C$. It differs from it by the sign $(-1)^i$ on the degree-$i$ component 
 (they agree for $\varepsilon=-1$ only). 
The sign is chosen by $\fu^{\mathrm s}=(-1)^i\fu$ and by the requirement 
 $\fp^\varepsilon=\fu^{\mathrm s}\theta^{\mathrm s}$ 
in the ribbon case. Therefore, we have $\dim_{\B}f=-\varepsilon$ instead of $+1$.

\subsection{Superalgebras}

\begin{definition}
\label{defsuperalgebra}
A \emph{superalgebra} in $\C$ is an algebra $(A,\mu,\eta)$ in $\C$ with a decomposition 
$A=A_0\oplus A_1$ (the grading is part of the data) such that $\eta$ factors through $A_0$ and 
$\mu(A_i\otimes A_j)\subset A_{i+j}$. It is \emph{supercommutative} if 
$\mu\circ c_{A,A}=\mu\circ\varsigma$, where $\varsigma\in\mathrm{End}(A\otimes A)$ acts 
as $(-1)^{ij}$ on $A_i\otimes A_j$.
\end{definition}

\begin{proposition}
\label{propbijection}
Set $\bA:=(A_0,A_1)\in\B$. Then: 
\begin{enumerate}

\item[(a)] $A\mapsto\bA$ is a bijection (on isomorphism classes) between superalgebras 
in $\C$ and algebras in $\B$, and $A$ is supercommutative iff $\bA$ is commutative in $\B$.

\item[(b)] Right $\bA$-modules in $\B$ are $\mathbb Z/2$-graded right $A$-modules 
$M=(M_0,M_1)$ in $\C$ with degree-preserving morphisms. $M$ is local iff, with 
the right action $a_M:M\otimes A\to M$ and the endomorphism $\mon:=c_{A,M}\circ c_{M,A}$ 
of $M\otimes A$, one has $a_M\circ\mon=a_M$.

\item[(c)] If $\C$ is ribbon: $\theta^{\mathrm s}_{\bA}=\id$ iff $\theta_{A_0}=\id$ and 
$\theta_{A_1}=\varepsilon\,\id$. If $\C$ is pivotal: 
$\fu^{\mathrm s}_{\bA}=\fp^\varepsilon_{\bA}$ iff $\fu_{A_0}=\fp_{A_0}$ and 
$\fu_{A_1}=\varepsilon\fp_{A_1}$.
\end{enumerate}
\end{proposition}

Note that in $(a)$ superalgebras are algebras with a given $\mathbb Z/2$-grading 
and graded isomorphisms, i.e.,  
 the statement is not about arbitrary ungraded algebra structures on $A$. 
 In $(b)$ we first have $c_{M,A}:M\otimes A\to A\otimes M$, then 
$c_{A,M}:A\otimes M\to M\otimes A$, both braidings of $\C$. 

\begin{proof}
(a) An algebra structure on an object of $\B$ is a graded algebra structure on the pair. 
On $A_i\otimes A_j$ the equation $\mu\circ c^{\mathrm s}_{\bA,\bA}=\mu$ reads 
$(-1)^{ij}\mu c_{A_i,A_j}=\mu$, i.e., $\mu c=\mu\varsigma$.
(b) A right $\bA$-module is a graded object with $a_M:M\otimes\bA\to M$, i.e., 
$M_i\otimes A_j\to M_{i+j}$. The locality equation in $\B$ is 
$a_M\circ c^{\mathrm s}_{\bA,M}\circ c^{\mathrm s}_{M,\bA}=a_M$ on $M\otimes\bA$, 
and by Lemma \ref{lemmonodromy} it is the stated equation.
(c) $\theta^{\mathrm s}$ acts by $\theta_{A_0}$ on $A_0$ and $\varepsilon\theta_{A_1}$ 
on $A_1$; for $\fu^{\mathrm s}=\fp^\varepsilon$ use Lemma \ref{lemdualityB}: 
$(-1)^i\fu_{A_i}=(-\varepsilon)^i\fp_{A_i}$.
\end{proof}

\begin{definition}
\label{defprops}
$A$ is \emph{haploid} ($\Hom_\C(\unit,A_0)\cong\kk$), \emph{exact}, \emph{Frobenius}, 
\emph{central}, if $\bA$ is such as an algebra in $\B$ 
 (thus ``$A$ exact'' always means exactness of $\bA$ in $\B$). A Frobenius algebra in a pivotal FTC 
is \emph{symmetric} if its Nakayama automorphism is the identity \cite{SY26}. 
For a commutative Frobenius algebra in a braided category with twist, symmetric is equivalent to 
$\theta_{A}=\id$ \cite{SY26}, which is \cite{FFRS06} (Proposition \ref{mirda} (iii) below). 
Applied to $\bA\in\B$ with the twist $\theta^{\mathrm s}$: a commutative Frobenius algebra $\bA$ is symmetric 
iff $\theta^{\mathrm s}_{\bA}=\id$, i.e., (Proposition \ref{propbijection} (c)) iff 
$\theta_{A_0}=\id$ and $\theta_{A_1}=\varepsilon\id$. We use this equivalence as the criterion; the 
definition is the Nakayama one.
 \end{definition}

 The sign $\varepsilon$ is part of the data. If $A_1\neq0$ and $\theta_{A_1}$ is a scalar 
$\pm1$, then $\varepsilon:=\theta_{A_1}$ follows; if $A_1=0$ any $\varepsilon$ works 
and $\C_A^{\sloc}=\C_{A}^{\loc}\boxtimes\sVect^\varepsilon$.

\subsection{Preliminaries}
\label{secexternal}

For the convenience we list here preliminary results from \cite{SY26, CMSY25, Shi19, CSZ25}.  
 $\D$ denotes a braided monoidal category and $A$ a commutative algebra in $\D$. 
  The following proposition is due to \cite{SY26}. 

\begin{proposition}
\label{kostra}

(i) {\rm Rigidity:} if $\D$ is rigid with equalizers and coequalizers (more generally closed monoidal) and ${}_A\D_A$ is 
rigid, then $\D_A$ and $\D_A^{\loc}$ are rigid.

(ii) {\rm Pivotal:} if in addition $\D$ is rigid with a pivotal structure $\fp$ and $\fu_A=\fp_A$, then $\D_A$ and $\D_A^{\loc}$ 
are pivotal.

(iii) {\rm Ribbon:} let $\D$ be ribbon (twist $\theta$, $\fp=\fu\theta$) 
with equalizers and coequalizers, rigid, and let ${}_A\D_A$ 
be rigid and $\theta_A=\id$. Then $\theta$ is a twist of $\D_A^{\loc}$ 
(Theorem 3.15, \cite{KO02}). 
 The object $D(A)=A^*$ with the right action of \cite{SY26} belongs to $\D_A^{\loc}$ 
(Prop. 3.18, which 
states that this holds iff $\theta_A^2=\id$). 
$\D_A^{\loc}$ is ribbon with twist $\theta$ iff 
$D(A)\in(\D_A^{\loc})'$ (Theorem 3.21).
\end{proposition}

Next lemma was taken from \cite{EO04} (see also \cite{SY26}). 
\begin{lemma}
\label{factEO}
{\rm Exactness and rigidity:}  if $\D$ is an FTC, ${}_A\D_A$ is rigid iff $A$ is exact.
\end{lemma}

In {\cite{SY26} one finds 
\begin{proposition} 
\label{mirda} 
(i) {\rm Local modules over exact algebras:} 
if $\D$ is a braided FTC and $A$ is an exact, commutative, haploid algebra in $\D$,
 then $\D_A$ is an FTC 
and $\D_A^{\loc}$ is a braided FTC. If $\D$ is non-degenerate, so is $\D_A^{\loc}$ 
(see also \cite{CMSY25}). 

(ii) {\rm Modularity:} if $\D$ is an MTC and $A$ is a haploid, commutative, exact, symmetric Frobenius algebra in $\D$, then 
$\D_A^{\loc}$ is an MTC (a non-degenerate ribbon FTC), with the twist inherited from $\D$.

(iii) {\rm Symmetry criterion:}
Let $\D$ be braided with a twist $\theta$. A commutative Frobenius algebra $A$ in $\D$ is symmetric iff 
$\theta_A=\id_A$ (see also \cite{FFRS06}).
\end{proposition}

The next lemma describes centralizers \cite{Shi19}. 
\begin{lemma}
\label{factcentr}
For a braided FTC $\D$: $\D$ is non-degenerate iff it is factorizable iff $\D'\simeq\Vect$ (Theorem 1.1). For every 
tensor full subcategory $\mathcal S\subset\D$ one has 
$\FPdim(\mathcal S)\FPdim(\mathcal S')=\FPdim(\D)\FPdim(\mathcal S\cap\D')$ and $\mathcal S''=\mathcal S\vee\D'$; 
in particular, if $\D$ is non-degenerate, $\FPdim(\mathcal S)\FPdim(\mathcal S')=\FPdim(\D)$ and 
$\mathcal S''=\mathcal S$ (Theorem 4.9). Here $\mathcal S'=\mathcal Z_{(2)}(\mathcal S\subset\D)$.
\end{lemma}

The simplicity implies exactness \cite{CSZ25} (see also \cite{SY26}) due to 
\begin{lemma}
\label{factsimple}
An algebra in a finite tensor category which is simple 
(has no two-sided ideals other than $0$ and itself), 
or a direct product of simple algebras, is exact. 
\end{lemma}

In the fusion case \cite{CMSY25} one has 
\begin{lemma}
\label{factfusion}
If $A$ is a simple commutative algebra in a braided fusion category $\D$, then $\D_A$ is a fusion category 
and $\D_A^{\loc}$ is a braided fusion category.
\end{lemma}

For pointed braided categories one obtains \cite{JS93} (see also \cite{EGNO16}) 
\begin{lemma}
\label{factJS}
Braided pointed fusion categories with group of invertibles $H$ are classified, up to braided monoidal 
equivalence, by the quadratic form $q(h)=c_{h,h}$ on $H$. The form $q\equiv1$ corresponds to $\Vect_H$ with 
trivial associator and trivial braiding; for $H=\mathbb Z/2$, $q(h)=\pm1$ corresponds to trivial 
associator, and $q=-1$ is $\sVect$.
\end{lemma}

Witt equivalence \cite{SY26} ({\cite{DMNO13} for the semisimple case) is given by  
\begin{lemma}
\label{factWitt}
Non-degenerate braided FTCs $\C,\D$ are Witt equivalent if $\C\boxtimes\Z(\mathcal A)\simeq\D\boxtimes\Z(\mathcal B)$ 
for FTCs $\mathcal A,\mathcal B$. The set $\Wns$ of Witt classes is an abelian group under $\boxtimes$, written 
additively, with $[\C]+[\D]=[\C\boxtimes\D]$, zero $[\Vect]$ and inverse $-[\C]=[\C^{\mathrm{rev}}]$ (since 
$\C\boxtimes\C^{\mathrm{rev}}\simeq\Z(\C)$ for non-degenerate $\C$, \cite{Shi19}). If $\C$ is 
non-degenerate and $A\in\C$ is a haploid, commutative, exact algebra, then $\C_A^{\loc}$ is non-degenerate and 
$[\C_A^{\loc}]=[\C]$, \cite{SY26}.
\end{lemma}

\begin{lemma}
\label{lemcentralizer}
Let $\N$ be a non-degenerate braided FTC and $\mathcal S\subset\N$ a tensor full subcategory. 
Then:

(i) $\FPdim\mathcal S\cdot\FPdim\mathcal S'=\FPdim\N$; 

(ii) $\mathcal S''=\mathcal S$; 

(iii) the M\"uger center of 
$\mathcal S'$ equals the M\"uger center $\mathcal S\cap\mathcal S'$ of $\mathcal S$. In particular, if 
$\mathcal S=\langle\unit,\tilde f\rangle\simeq\mathsf{sVect}$ then $\mathcal S$ is symmetric, $\mathcal S\subset\mathcal S'$ and the 
M\"uger center of $\mathcal S'$ is $\mathcal S$.
\end{lemma}

\begin{proof}
(i), (ii) are Lemma \ref{factcentr} with $\N'=\Vect$. (iii) The M\"uger center of a tensor full subcategory 
$\mathcal E$ is $\mathcal E\cap\mathcal E'$. For $\mathcal E=\mathcal S'$ this is $\mathcal S'\cap\mathcal S''=\mathcal S'\cap\mathcal S$ 
by (ii). If $\mathcal S$ is symmetric then $\mathcal S\subset\mathcal S'$, so $\mathcal S'\cap\mathcal S=\mathcal S$.
\end{proof}

\section{Rigidity, pivotality and ribbon structure}
\label{secrigid}

\begin{proof}[Proof of Theorem \ref{thmAs}]
$\B$ is an additive rigid braided category (Section \ref{secgraded} and Lemma \ref{lemzigzag}), hence 
closed monoidal, and it has equalizers and coequalizers (they are computed componentwise, Section 
\ref{secgraded}). $\bA$ is a commutative algebra in $\B$ according to Proposition  
\ref{propbijection} (a). Next,  $\C_A^{\sloc}=\B_{\bA}^{\loc}$ due to Proposition  
\ref{propbijection} (b). We apply the results of \cite{SY26} mentioned in Section \ref{secexternal}.

(a) Proposition \ref{kostra} (i) applied to $(\B,\bA)$: if ${}_{\bA}\B_{\bA}$ is rigid, then $\B_{\bA}^{\loc}$ is rigid.

(b) The pivotal structure on $\B$ is $\fp^\varepsilon$ (Lemma \ref{lemdualityB} (ii)), the Drinfeld 
isomorphism of $\B$ is $\fu^{\mathrm s}$ (Lemma \ref{lemdualityB} (i)), and the condition 
$\fu^{\mathrm s}_{\bA}=\fp^\varepsilon_{\bA}$ of Proposition \ref{kostra} (ii) is equivalent to 
$\fu_{A_0}=\fp_{A_0}$, $\fu_{A_1}=\varepsilon\fp_{A_1}$ by Proposition \ref{propbijection} (c).

(c) $\B$ is ribbon with twist $\theta^{\mathrm s}$ (Section \ref{secgraded}), and 
$\fp^\varepsilon=\fu^{\mathrm s}\theta^{\mathrm s}$ (Lemma \ref{lemdualityB} (iii)). The condition 
$\theta^{\mathrm s}_{\bA}=\id$ of Proposition \ref{kostra} (iii) is equivalent to $\theta_{A_0}=\id$, 
$\theta_{A_1}=\varepsilon\id$ (Proposition \ref{propbijection} (c)). By Proposition \ref{kostra} (iii), 
$\theta^{\mathrm s}$ is then a twist of $\B^{\loc}_{\bA}$, and $\B^{\loc}_{\bA}$ is ribbon iff 
$\bA^*\in(\B^{\loc}_{\bA})'$, which is the remaining assumption. \qedhere
\end{proof}

Note that we use only that $\bA$ is an ordinary commutative algebra of $\B$.  
The condition $\theta^2=\id$ of \cite{SY26}, resp. \cite{CMSY25},  
determines only whether $\bA^*$ (resp.\ the dualizing object $D\bA$) is local. The transparency of $\bA^*$ 
required in (c) is a stronger condition, equivalent to the compatibility of $\theta^{\mathrm s}$ with duality 
(Remark \ref{remAstar}). In our setting $(\theta^{\mathrm s}_{\bA})^2=\id$ means $\theta_{A_0}^2=\id$ and 
$\theta_{A_1}^2=\id$. 

\section{Finite tensor categories and super-modularity}

\subsection{The toric code}
\label{sectoric}

Let $\T=\Z(\Vect_{\mathbb Z/2})$ be the toric code, regarded as a braided fusion category: 
$\Inv(\T)=\{\unit,e,m,\psi\}$ with $\psi=e\otimes m$ and group $\mathbb Z/2\times\mathbb Z/2$, trivial 
associator (its underlying fusion category is $\Vect_{\mathbb Z/2\times\mathbb Z/2}$ with trivial 
$3$-cocycle), $c_{e,e}=c_{m,m}=\id$, $\mon_{e,m}=\mon_{\psi,e}=\mon_{\psi,m}=-1$, 
$\mon_{e,e}=\mon_{m,m}=1$; hence $c_{\psi,\psi}=-\id$. It is semisimple and non-degenerate 
($\T'=\langle\unit\rangle$, since each of $e,m,\psi$ has monodromy $-1$ with some simple object). 
Its standard ribbon structure is $\theta_e=\theta_m=1$, $\theta_\psi=-1$. 

For $\varepsilon\in\{\pm1\}$ define the ribbon category $\T^\varepsilon$ as follows. For $\varepsilon=-1$ 
it is $\T$ with the standard twist. For $\varepsilon=+1$ set $\theta'_X:=\theta_X\mon_{e,X}$. Since 
monodromy is bimultiplicative on invertible objects, $X\mapsto\mon_{e,X}$ is a $\pm1$-valued character. 
 Thus, $\theta'_{X\otimes Y}=\theta_X\theta_Y\mon_{X,Y}\mon_{e,X}\mon_{e,Y}=\theta'_X
\theta'_Y\mon_{X,Y}$ and $\theta'_{X^*}=\theta'_X$. Therefore, $\theta'$ is a ribbon structure, 
with $\theta'_e=1,\ \theta'_m=-1,\ \theta'_\psi=+1$ and $\dim\psi=\dim m=-1$ (by Remark 
\ref{remtheta-vs-c}, since $q(m)=1$ and $q(\psi)=-1$). So $\T^{+}$ is not pseudo-unitary (some 
dimensions are negative); it is nevertheless spherical, as all simple objects are self-dual 
(see Remark \ref{remsphericity}).
The underlying braided category of $\T^+$ and $\T^-$ is the same, but the ribbon 
structure (hence pivotal structure and dimensions) differs (both are non-degenerate). 
In both, $\langle\unit,\psi\rangle\simeq\sVect^\varepsilon$ is a full ribbon subcategory equal to its 
centralizer.

\begin{lemma}
\label{lemtoric}
Let $\C$ be a braided FTC and $\widetilde\B:=\C\boxtimes\T^\varepsilon$. Choose an isomorphism 
$m_\psi:\psi\otimes\psi\to\unit$ with $m_\psi\otimes\id_\psi=\id_\psi\otimes m_\psi$ (possible since the associator 
of $\langle\unit,\psi\rangle$ is trivial; see below). Then $\iota_0:\sVect\to\T^\varepsilon$, $\unit\mapsto\unit$, $f\mapsto\psi$ 
with tensor structure given by $m_\psi$ on $f\otimes f$ and by identities otherwise, is a braided tensor 
functor, ribbon for $\sVect^\varepsilon$, and 
\[
\iota:=\id_\C\boxtimes\iota_0:\ \B=\C\boxtimes\sVect\to\widetilde\B,\qquad (X_0,X_1)\mapsto X_0\boxtimes\unit\oplus X_1\boxtimes\psi,
\]
is a fully faithful braided tensor functor, ribbon if $\C$ is ribbon, and 
$\iota(\B)=\mathcal Z_{(2)}(\unit\boxtimes\psi\subset\widetilde\B)$. If $\C$ is ribbon, $\iota$ preserves the pivotal 
structures $\fp^\varepsilon=\fu^{\mathrm s}\theta^{\mathrm s}$ and $\widetilde\fp=\widetilde\fu\widetilde\theta$ (with 
$\widetilde\fp_\psi=-\varepsilon$).
\end{lemma}

\begin{proof}
\emph{Tensor structure.} For $\iota_0$ to be a tensor functor the only axiom with non-trivial content is the 
hexagon on $f\otimes f\otimes f$, which reads $m_\psi\otimes\id_\psi=\id_\psi\otimes m_\psi$ (the associators of 
$\sVect$ and of $\langle\unit,\psi\rangle$ being trivial). Both sides belong to the one-dimensional space 
$\Hom(\psi^{\otimes3},\psi)$, so they differ by a scalar $\lambda$; $\lambda=\pm1$ is the value 
of the associator on $(\psi,\psi,\psi)$, which does not change under rescaling $m_\psi$ and 
is $1$ for the trivial associator of $\T$. Hence, $m_\psi$ as required exists, and $\iota_0$ is a tensor functor. 
It is braided: the non-trivial condition is $\iota_0(c_{f,f})m_\psi=m_\psi c_{\psi,\psi}$, i.e., $(-1)m_\psi=m_\psi(-1)$. 
It is ribbon for $\sVect^\varepsilon$ since $\theta_\psi=\varepsilon$. Hence, $\iota=\id_\C\boxtimes\iota_0$ is a braided 
(ribbon) tensor functor, and its coherence constraints are those of $\iota_0$ with $\id_\C$: for $X,Y\in\B$ the 
tensorator $\iota(X)\otimes\iota(Y)\to\iota(X\otimes Y)$ is $\id_{X_i\otimes Y_j}\boxtimes J_{i,j}$ on the summand 
$X_i\boxtimes\psi^i\otimes Y_j\boxtimes\psi^j$, where $J_{1,1}=m_\psi$ and $J_{i,j}=\id$ otherwise. 
Different choices of $m_\psi$ differ by a non-zero scalar and give isomorphic functors. 

\emph{Fully faithful.}
$\Hom_{\C\boxtimes\T}(X_0\boxtimes\unit\oplus X_1\boxtimes\psi,Y_0\boxtimes\unit
\oplus Y_1\boxtimes\psi)=\Hom(X_0,Y_0)\oplus\Hom(X_1,Y_1)$ since $\Hom_\T(\unit,\psi)=0$.

\emph{Braided and ribbon.} Explicitly, 
$c_{X\boxtimes\psi,Y\boxtimes\psi}=c_{X,Y}\boxtimes c_{\psi,\psi}=-c_{X,Y}
=c^{\mathrm s}$ on odd-odd, and $c_{X\boxtimes\psi,Y\boxtimes\unit}=c_{X,Y}
\boxtimes c_{\psi,\unit}=c_{X,Y}$. 
It is ribbon since $\theta_{X\boxtimes\psi}=\theta_X\boxtimes\varepsilon=\theta^{\mathrm s}$. 
As $\iota$ is braided and compatible with duals, it preserves the Drinfeld isomorphisms, and together with 
the twists it preserves $\fp=\fu\theta$. 
 On $\unit\boxtimes\psi$ this gives $\widetilde\fp_\psi=\widetilde\fu_\psi\theta_\psi=(-1)\varepsilon$. 

\emph{Centralizer.} $\T$ is pointed, thus every object of $\widetilde\B$ is canonically 
$\bigoplus_{t\in\Inv(\T)}X_t\boxtimes t$. The monodromy of $X_t\boxtimes t$ with 
$\unit\boxtimes\psi$ is $\mon_{t,\psi}\cdot\id$, with $\mon_{t,\psi}=-1$ for 
$t\in\{e,m\}$. It is the identity iff $X_e=X_m=0$.
\end{proof}

\subsection{Exactness of the image of \texorpdfstring{$\bA$}{A-bar}}
\label{secexactness}

The following lemma supplies the step on which Theorem \ref{thmBs} (1), (2) depend. 
 Exactness of an algebra 
is not preserved by arbitrary tensor functors, and $\iota$ is far from surjective, 
thus the statement is proven using the structure of $\T$. 

\begin{lemma}
\label{lemexact}
Let $\C$ be a braided FTC and let $\bA$ be a haploid commutative exact algebra in $\B=\C^{\gr}$. Then 
$A':=\iota(\bA)$ is a haploid commutative exact algebra in $\widetilde\B=\C\boxtimes\T^\varepsilon$.
\end{lemma}

\begin{proof}
\emph{Commutative:} $\iota$ is a braided tensor functor, 
thus $\mu'\circ c_{A',A'}=\iota(\mu\circ c^{\mathrm s})=\mu'$. 

\emph{Haploid:} $\Hom_{\widetilde\B}(\unit,A')=\Hom_\C(\unit,A_0)\otimes\Hom_\T(\unit,\unit)\oplus
\Hom_\C(\unit,A_1)\otimes\Hom_\T(\unit,\psi)=\Hom_\C(\unit,A_0)\cong\kk$. 

\emph{Exact:} By \cite{EO04}, $A'$ is exact iff the category $\widetilde\B_{A'}$ of right $A'$-modules is an 
exact left $\widetilde\B$-module category, i.e., iff $\widetilde P\otimes M$ is projective in $\widetilde\B_{A'}$ for every 
projective $\widetilde P\in\widetilde\B$ and every $M\in\widetilde\B_{A'}$ \cite{EGNO16}. Likewise for $\bA$ and 
$\B_{\bA}$. As $\T$ is semisimple, $\widetilde\B=\bigoplus_{t\in\Inv(\T)}\C\boxtimes t$ as an abelian category, and the 
projective objects are the sums of objects $P\boxtimes t$ with $P$ projective in $\C$. 

Set $\T_1=\{\unit,\psi\}$ and $\T_e=\{e,m\}$. The algebra $A'$ belongs to $\C\boxtimes\T_1$, and tensoring on the right with 
$\unit$ or $\psi$ preserves $\T_1$ and $\T_e$ (since $e\otimes\psi=m$, $m\otimes\psi=e$). 
 Hence, every right $A'$-module 
decomposes canonically as $M=M_1\oplus M_e$ ($M_1$ supported on $\C\boxtimes\T_1$, $M_e$ on $\C\boxtimes\T_e$), and 
$\widetilde\B_{A'}=(\widetilde\B_{A'})_1\oplus(\widetilde\B_{A'})_e$ as abelian categories.
The functor $\iota$ identifies $\B$ with the full subcategory $\C\boxtimes\T_1$ and, being fully faithful and 
monoidal, identifies $(\widetilde\B_{A'})_1$ with $\B_{\bA}$. For the second summand consider 
$E:\B\to\C\boxtimes\T_e$, $E(Y_0,Y_1)=Y_0\boxtimes e\oplus Y_1\boxtimes m$. Using $e\otimes\unit=e$, 
$e\otimes\psi=m$, $m\otimes\psi=e$ one has $E(Y)\otimes\iota(Z)\cong E(Y\otimes Z)$, and since the 
associator of $\T$ is trivial these isomorphisms can be chosen coherently. 
 Thus, $E$ is an equivalence of 
right $\B$-module categories. 
 $\T_e$ is a free orbit of $\langle\unit,\psi\rangle\simeq\Vect_{\mathbb Z/2}$, 
so $\C\boxtimes\T_e$ is the regular module category $\B$.
  Consequently, $(\widetilde\B_{A'})_e\simeq\B_{\bA}$ as well, and 
$\widetilde\B_{A'}\simeq\B_{\bA}\oplus\B_{\bA}$, an object being projective iff both components are.

Let $\widetilde P=P\boxtimes t$ with $P$ projective in $\C$ and $t\in\{\unit,e,m,\psi\}$. Then $(P,0)\in\B$ is projective. A 
 computation with the fusion rules of $\T$ shows that under the identifications above, left multiplication 
by $\widetilde P$ is, on each of the two copies of $\B_{\bA}$, isomorphic to 
$N\mapsto\Pi^a\big((P,0)\otimes N\big)$ followed by a possible interchange of the two copies. Explicitly, for 
$Y=(Y_0,Y_1)$: 
\[
(P\boxtimes\unit)\otimes\iota(Y)=\iota((P,0)\otimes Y),\qquad
(P\boxtimes\psi)\otimes\iota(Y)=\iota\big((0,P)\otimes Y\big)=\iota\big(\Pi((P,0)\otimes Y)\big),
\]
\[
(P\boxtimes e)\otimes\iota(Y)=E((P,0)\otimes Y),\qquad (P\boxtimes m)\otimes\iota(Y)=E\big(\Pi((P,0)\otimes Y)\big),
\]
\[
(P\boxtimes e)\otimes E(Y)=\iota((P,0)\otimes Y),\qquad (P\boxtimes m)\otimes E(Y)=\iota\big(\Pi((P,0)\otimes Y)\big),
\]
and $(P\boxtimes\unit)\otimes E(Y)=E((P,0)\otimes Y)$, $(P\boxtimes\psi)\otimes E(Y)=E(\Pi((P,0)\otimes Y))$, where 
$\Pi=f\otimes-$ is the parity shift. 
 The identities hold as identities of $\bA$-modules, the action being 
$\id_P\otimes a$. Since $\bA$ is exact in $\B$ and $(P,0)$ is projective in $\B$, $(P,0)\otimes N$ is projective in 
$\B_{\bA}$ for all $N\in\B_{\bA}$, and $\Pi$ is an autoequivalence of $\B_{\bA}$. Hence $\widetilde P\otimes M$ is projective in 
$\widetilde\B_{A'}$ for all $M$. 
\end{proof}

\subsection{The induction lemma}
\label{secinduction}

Let $\D$ be a braided monoidal category with coequalizers preserved by $\otimes$ in 
each variable (e.g., a finite tensor category), and $A$ a commutative algebra in $\D$. For a right module $N$ set 
$\lambda_N:=a_N\circ c_{A,N}:A\otimes N\to N$ (a left action). For a right module 
$M$ and a right module $N$ with $\lambda_N$, $M\otimes_AN$ is the coequalizer of 
$a_M\otimes\id_N$ and $\id_M\otimes\lambda_N$ on 
$M\otimes A\otimes N\rightrightarrows M\otimes N$. Denote by $\pi_{M,N}$ the projection. 
For local $N$, $a_Nc_{A,N}=a_Nc_{N,A}^{-1}$ (apply $c_{N,A}^{-1}$ to $a_N\mon=a_N$). 
 Thus, the left action does not depend on the choice between $c$ and $c^{-1}$ and 
$\D_A^{\loc}$ (with $\otimes_A$) is the same as in \cite{Par95, KO02, CMSY25} 
(where left modules are used; see Section \ref{seccomparison} for the correspondence). 
For local $M$, $N$ the braiding 
$c^A_{M,N}:M\otimes_AN\to N\otimes_AM$ is the map induced by $c_{M,N}$, i.e., 
$c^A_{M,N}\pi_{M,N}=\pi_{N,M}c_{M,N}$ \cite{Par95}.

Let $Z_A:=\{X\in\D\mid c_{A,X}\circ c_{X,A}=\id_{X\otimes A}\}$ (equivalently 
$c_{X,A}c_{A,X}=\id_{A\otimes X}$), $F(X):=X\otimes A$ with right action $\id_X\otimes\mu$.

\begin{lemma}
\label{leminduction}
Let $X$, $Y\in Z_A$ and let $M$ be a local right $A$-module. Then: 
\begin{enumerate}

\item[(i)] $\unit\in Z_A$, $X\otimes Y\in Z_A$, and $F(X)$ is local.

\item[(ii)] There are  isomorphisms $\rho_{X,M}:F(X)\otimes_AM\to X\otimes M$ 
(valid for every right module $M$; an isomorphism of right $A$-modules, $X\otimes M$ carrying the action 
$\id_X\otimes a_M$) and $\sigma_{M,X}:M\otimes_AF(X)\to M\otimes X$ (an isomorphism in $\D$), 
given by $\rho\pi=\id_X\otimes\lambda_M$ and 
$\sigma\pi=(a_M\otimes\id_X)(\id_M\otimes c_{X,A})$. 
Moreover $\rho_{X,F(Y)}:F(X)\otimes_AF(Y)\to X\otimes Y\otimes A=F(X\otimes Y)$ 
is an isomorphism of right modules.

\item[(iii)] $\sigma_{M,X}\,c^A_{F(X),M}\,\rho_{X,M}^{-1}=c_{X,M}$ and 
$\rho_{X,M}\,c^A_{M,F(X)}\,\sigma_{M,X}^{-1}=c_{M,X}$. Consequently, under $\sigma$ 
the monodromy $c^A_{F(X),M}c^A_{M,F(X)}$ of $M$ and $F(X)$ in $\D_A^{\loc}$ is the 
monodromy $c_{X,M}c_{M,X}$ of $M$ and $X$ in $\D$ (as an endomorphism of $M\otimes X$).  
 In particular, it is the identity iff $\mon_{M,X}=\id_{M\otimes X}$.

\item[(iv)] The composite
\[
F(X\otimes Y)=X\otimes Y\otimes A\xrightarrow{\ \rho_{X,F(Y)}^{-1}\ }F(X)\otimes_AF(Y)
\xrightarrow{\ c^A_{F(X),F(Y)}\ }F(Y)\otimes_AF(X)\xrightarrow{\ \rho_{Y,F(X)}\ }Y\otimes X\otimes A=F(Y\otimes X)
\]
equals $c_{X,Y}\otimes\id_A$.

\item[(v)] If $\D$ is ribbon and $\theta_A=\id$, then $\theta_{F(X)}=\theta_X\otimes\id_A$.
\end{enumerate}
\end{lemma}

The isomorphisms $\rho_{X,F(Y)}$ in (ii) are the components of the well-known monoidal structure of the 
induction functor $F$ \cite{Par95, KO02}.
   The proof of Theorem \ref{thmBs} uses that $\rho_{\psi,F(\psi)}$ is an isomorphism and 
the identities (iii)-(v), which are proved below from the definitions.

\begin{proof}
We suppress associators and unitors. 

(i) The hexagon axioms give $c_{X\otimes A,A}=(c_{X,A}\otimes\id_A)(\id_X\otimes c_{A,A})$ 
and $c_{A,X\otimes A}=(\id_X\otimes c_{A,A})(c_{A,X}\otimes\id_A)$, hence the endomorphism 
$\mon_{F(X),A}=c_{A,F(X)}\,c_{F(X),A}$ of $F(X)\otimes A$ is
\[
\mon_{F(X),A}=c_{A,X\otimes A}\,c_{X\otimes A,A}=(\id_X\otimes c_{A,A})\,
\big((c_{A,X}c_{X,A})\otimes\id_A\big)\,(\id_X\otimes c_{A,A}).
\]
For $X\in Z_A$ the middle factor is the identity, thus $\mon_{F(X),A}=\id_X\otimes c_{A,A}^2$, 
and $a_{F(X)}\mon_{F(X),A}=\id_X\otimes\mu c_{A,A}c_{A,A}=\id_X\otimes\mu=a_{F(X)}$ by commutativity 
applied twice. The same hexagon computation shows $Z_A$ is closed under $\otimes$ and contains 
$\unit$.

(ii) For $\rho$, apply $X\otimes-$ to the diagram 
$A\otimes A\otimes M\rightrightarrows A\otimes M\xrightarrow{\lambda_M}M$ with the two maps 
$\mu\otimes\id_M$ and $\id_A\otimes\lambda_M$. This is a split coequalizer, with sections 
$t=\eta\otimes\id_M:M\to A\otimes M$ and 
$u=\eta\otimes\id_A\otimes\id_M:A\otimes M\to A\otimes A\otimes M$: indeed $\lambda_Mt=\id$, 
$(\mu\otimes\id)u=\id$, $(\id_A\otimes\lambda_M)u=t\lambda_M$, and 
$\lambda_M(\mu\otimes\id_M)=\lambda_M(\id_A\otimes\lambda_M)$ by the module axiom. 
Split coequalizers are absolute, thus $X\otimes-$ preserves it, and $X\otimes M$ with 
projection $\id_X\otimes\lambda_M$ is the coequalizer defining $F(X)\otimes_AM$. $\rho$ 
is a right module map for the action $\id_X\otimes a_M$ since 
$\lambda_M(\id_A\otimes a_M)=a_M(\lambda_M\otimes\id_A)$; explicitly, by naturality of $c_{A,-}$, 
associativity of the module and commutativity of $A$,
\begin{align*}
\lambda_M(\id_A\otimes a_M)&=a_M\,c_{A,M}(\id_A\otimes a_M)=a_M(a_M\otimes\id_A)\,c_{A,M\otimes A}
=a_M(\id_M\otimes\mu)(\id_M\otimes c_{A,A})(c_{A,M}\otimes\id_A)\\
&=a_M(\id_M\otimes\mu)(c_{A,M}\otimes\id_A)=a_M(a_M\otimes\id_A)(c_{A,M}\otimes\id_A)
=a_M(\lambda_M\otimes\id_A).
\end{align*}

For $\sigma$,  the left action on $F(X)$ is 
$\lambda_{F(X)}=a_{F(X)}c_{A,F(X)}=(\id_X\otimes\mu)(c_{A,X}\otimes\id_A)$ 
(using $\mu c_{A,A}=\mu$). For $X\in Z_A$ one has $c_{X,A}^{-1}=c_{A,X}$, and then by 
naturality of $c_{X,-}$ with respect to $\mu$ and the hexagon, 
$c_{X,A}\,\lambda_{F(X)}\,(\id_A\otimes c_{X,A}^{-1})=\mu\otimes\id_X$ 
(indeed $c_{X,A}(\id_X\otimes\mu)=(\mu\otimes\id_X)(\id_A\otimes c_{X,A})(c_{X,A}\otimes\id_A)$, and 
$c_{X,A}c_{A,X}=\id$). 
Thus, $c_{X,A}:F(X)\to(A\otimes X,\mu\otimes\id_X)$ transports $\lambda_{F(X)}$ to the 
free left module structure, and
$M\otimes_AF(X)\cong M\otimes_A(A\otimes X)\cong M\otimes X$ via the split coequalizer 
$M\otimes A\otimes A\otimes X\rightrightarrows M\otimes A\otimes X
\xrightarrow{a_M\otimes\id_X}M\otimes X$ (maps $\id_M\otimes\mu\otimes\id_X$ and 
$a_M\otimes\id_A\otimes\id_X$; sections $\id_M\otimes\eta\otimes\id_X$ and 
$\id_{M\otimes A}\otimes\eta\otimes\id_X$). The composite is $\sigma$, with 
$\sigma\pi=(a_M\otimes\id_X)(\id_M\otimes c_{X,A})$. The statement about $F(X)\otimes_AF(Y)$ 
is $\rho_{X,F(Y)}$ (take $M=F(Y)$ in the first part, with $X\otimes F(Y)=X\otimes Y\otimes A$).

(iii) For inverses, $\rho^{-1}=\pi_{X,M}(\id_X\otimes\eta\otimes\id_M)$ and 
$\sigma^{-1}=\pi_{M,X}(\id_M\otimes\id_X\otimes\eta)$. Indeed 
$\rho\pi(\id\otimes\eta\otimes\id)=\id_X\otimes\lambda_M(\eta\otimes\id_M)=\id$ and 
$\sigma\pi(\id_{M\otimes X}\otimes\eta)=(a_M\otimes\id_X)(\id_M
\otimes c_{X,A}(\id_X\otimes\eta))=(a_M\otimes\id_X)(\id_M\otimes\eta\otimes\id_X)=\id$ 
(naturality of $c_{X,-}$ with respect to $\eta$). Now compute, using $c^A\pi=\pi c$ and 
$c_{A,M}(\eta\otimes\id_M)=\id_M\otimes\eta$:
\[
c_{X\otimes A,M}(\id_X\otimes\eta\otimes\id_M)=(c_{X,M}\otimes\id_A)(\id_X
\otimes c_{A,M})(\id_X\otimes\eta\otimes\id_M)=(\id_{M\otimes X}\otimes\eta)\,c_{X,M},
\]
 thus $\sigma c^A_{F(X),M}\rho^{-1}=\sigma\pi_{M,X}(\id_{M\otimes X}\otimes\eta)c_{X,M}
=c_{X,M}$. Similarly, 
\[
c_{M,X\otimes A}(\id_{M\otimes X}\otimes\eta)=(\id_X\otimes c_{M,A})(c_{M,X}\otimes\eta)
=(\id_X\otimes\eta\otimes\id_M)\,c_{M,X},
\]
thus $\rho c^A_{M,F(X)}\sigma^{-1}=\rho\pi_{X,M}(\id\otimes\eta\otimes\id)c_{M,X}=c_{M,X}$. 
The monodromy statement follows by composing the two identities. Note that locality of $M$ and $F(X)$ 
enters through the existence of $c^A$.

(iv) By (iii) with $M=F(Y)$: $\sigma_{F(Y),X}c^A_{F(X),F(Y)}\rho^{-1}_{X,F(Y)}=c_{X,F(Y)}$, a morphism 
$X\otimes F(Y)\to F(Y)\otimes X$ (note $F(Y)$ is local by (i)). Also,  
$\rho_{Y,F(X)}\sigma^{-1}_{F(Y),X}=(\id_Y\otimes\lambda_{F(X)})(\id_{Y\otimes A\otimes X}\otimes\eta)=\id_Y\otimes c_{A,X}$, 
a morphism $Y\otimes A\otimes X\to Y\otimes X\otimes A$. Hence, the composite in (iv) equals 
$(\id_Y\otimes c_{A,X})\,c_{X,Y\otimes A}=(\id_Y\otimes c_{A,X})(\id_Y\otimes c_{X,A})(c_{X,Y}\otimes\id_A)
=c_{X,Y}\otimes\id_A$, since $c_{A,X}c_{X,A}=\id$ on $X\otimes A$.

(v) The twist of $\D_A^{\loc}$ is $\theta_M$ when $\theta_A=\id$ (Proposition \ref{kostra} (iii)). 
Balancing: $\theta_{X\otimes A}=(\theta_X\otimes\theta_A)\mon_{X,A}=\theta_X\otimes\id_A$ for $X\in Z_A$.
\end{proof}

\subsection{Proof of Theorem \ref{thmBs}}
\label{secproofB}

\begin{proof}
 0) $\B=\C^{\gr}$ is a braided FTC (see Section \ref{secgraded}), 
 and $\bA$ is a haploid commutative exact algebra in 
$\B$ (Proposition \ref{propbijection} and Definition \ref{defprops}). Hence,
 $\B_{\bA}^{\loc}=\C_A^{\sloc}$ is a braided FTC by 
 Proposition \ref{mirda} (i). Moreover ${}_{\bA}\B_{\bA}$ is rigid by Lemma \ref{factEO}, 
thus the rigidity assumption in Theorem 
\ref{thmAs} holds, and the remaining statements are those of Theorem \ref{thmAs}.

1) \emph{The ambient category.} Let $\widetilde\B:=\C\boxtimes\T^\varepsilon$. It is non-degenerate: since $\T$ is 
semisimple and pointed, an object $M=\bigoplus_tX_t\boxtimes t$ is transparent iff 
each non-zero summand is. 
 The monodromy of $X_t\boxtimes t$ with $\unit\boxtimes s$ is $\mon_{t,s}$, 
 thus $X_t\neq0$ leads to $t\in\T'=\langle\unit\rangle$, 
and then $X_\unit$ must be transparent in $\C$, i.e., $X_\unit\in\C'\simeq\Vect$. 
Thus, $\widetilde\B'\simeq\Vect$, and 
$\widetilde\B$ is non-degenerate (see Lemma \ref{factcentr}).

\emph{The algebra.} Set $A':=\iota(\bA)\in\widetilde\B$. By Lemma \ref{lemexact} it is haploid, commutative and exact. 
Hence, $\N:=\widetilde\B_{A'}^{\loc}$ is a non-degenerate braided FTC by Proposition \ref{mirda} (i).

\emph{The fermion.} $A'\in\iota(\B)=\mathcal Z_{(2)}(\unit\boxtimes\psi)$ 
 (see Lemma \ref{lemtoric}). Thus, 
$\psi:=\unit\boxtimes\psi$ centralizes $A'$, i.e., 
$\psi\in Z_{A'}$. Set $\tilde f:=F(\psi)=\psi\otimes A'$. 
By Lemma \ref{leminduction} (i) $\tilde f$ is local. By (ii) and $\psi\otimes\psi\cong\unit$, 
$\tilde f\otimes_{A'}\tilde f\cong F(\psi\otimes\psi)\cong A'=\unit_\N$. 
 Thus, $\tilde f$ is invertible, hence 
simple ($A'$ is haploid, so the unit of $\N$ is simple). By (iv) with $X=Y=\psi$, the braiding $c^{A'}_{\tilde f,\tilde f}$ 
corresponds to $c_{\psi,\psi}\otimes\id_{A'}=-\id$ on $\tilde f\otimes_{A'}\tilde f\cong A'$. Since the braiding of the unit object with 
itself is the identity, $\tilde f\not\cong\unit_\N$, thus $\langle\unit_\N,\tilde f\rangle\simeq\sVect$ as braided categories 
(see Lemma \ref{factJS}, $q(\tilde f)=-1$). 
As a right $A'$-module $\tilde f=\psi\otimes A'$ with action $\id_\psi\otimes\mu$.
  Under $\iota$ this is 
$\iota(f\otimes\bA)$ where $f\otimes\bA=(A_1,A_0)=:\Pi\bA$ with action $\id_f\otimes\mu_{\bA}$, i.e., the module with graded 
object $(A_1,A_0)$ and the action of $A_i$ on $A_j$ given by $\mu$. Since $\iota$ is a tensor functor, 
$\iota(\Pi\bA)\cong\tilde f$ as right $A'$-modules.

\emph{Identification of $\C_A^{\sloc}$.} $\iota$ is fully faithful, thus right $\bA$-modules in $\B$ are the right $A'$-modules 
in $\widetilde\B$ whose underlying object belongs to $\iota(\B)$ (the action map is a morphism between objects of 
$\iota(\B)$, hence in $\iota(\B)$), with the same locality equation. Thus, $\B_{\bA}^{\loc}$ is the full subcategory of 
$\N$ of local modules with underlying object in $\iota(\B)$. Let $M\in\N$. By Lemma \ref{leminduction} (iii) (with 
$\D=\widetilde\B$, $X=\psi$) the monodromy of $M$ and $\tilde f$ in $\N$ is conjugate, via the isomorphism 
$\sigma_{M,\psi}$, to the monodromy $\mon_{M,\psi}$ of the underlying object of $M$ with $\psi$ in $\widetilde\B$. 
Therefore, $M\in\mathcal Z_{(2)}(\langle\tilde f\rangle\subset\N)$, which is equivalent to $\mon^{\N}_{M,\tilde f}=\id$, 
equivalently to $\mon_{M,\psi}=\id_{M\otimes\psi}$, equivalently to (Lemma \ref{lemtoric}) the underlying object of 
$M$ belonging to $\iota(\B)$, which is equivalent to $M\in\B_{\bA}^{\loc}$. 
 Note that each step is an 
equivalence: $\sigma$ is an isomorphism, and since $\langle\tilde f\rangle$ is generated by $\tilde f$ the centralizer 
condition is the single condition on $\tilde f$. Thus, $\C_A^{\sloc}=\D:=\mathcal Z_{(2)}(\langle\tilde f\rangle\subset\N)$.

\emph{M\"uger center and dimension.} $\mathcal S:=\langle\unit_\N,\tilde f\rangle$ (objects isomorphic to sums of 
$\unit_\N$ and $\tilde f$) is a symmetric tensor full subcategory of $\N$, equivalent to $\sVect$, with $\FPdim\mathcal S=2$. 
By Lemma \ref{lemcentralizer}, the M\"uger center of $\D=\mathcal S'$ is $\mathcal S$, and 
$\FPdim\D=\FPdim\N/2$. Thus, $\D=\C_A^{\sloc}$ is super-non-degenerate with fermion $\tilde f$ (see Definition \ref{defsupernd}). 

 2) Let $\C$ be ribbon and $\bA$ symmetric Frobenius. Then $\B$ (with $\theta^{\mathrm s}$), $\T^\varepsilon$ and 
$\widetilde\B$ are ribbon, $\widetilde\B$ is an MTC, and $\iota$ is a ribbon functor preserving the pivotal structures 
(Lemma \ref{lemtoric}). Then $A'$ is Frobenius (a strong monoidal functor preserves the algebra, coalgebra and 
Frobenius relations), commutative, and 
$\widetilde\theta_{A'}=\iota(\theta^{\mathrm s}_{\bA})=\id$. Hence $A'$ is symmetric Frobenius 
(Proposition \ref{mirda} (iii)), 
and, being haploid, commutative and exact (Lemma \ref{lemexact}), $\N$ is an MTC (Proposition \ref{mirda} (ii)).

We distinguish three ribbon structures: 

(a) $\N$ carries the twist $\theta^{\N}_M=\widetilde\theta_M$ of the underlying 
object (Proposition \ref{kostra} (iii)), for which $\N$ is an MTC. 

(b) $\D\subset\N$ is a full subcategory closed under tensor 
products (it is a centralizer) and duals, hence inherits the ribbon structure $\theta^{\N}|_\D$. 

(c) Under the 
identification $\D=\B_{\bA}^{\loc}$ of part (1), $\theta^{\N}_M=\widetilde\theta_{\iota(M')}=\iota(\theta^{\mathrm s}_{M'})$ 
for $M=\iota(M')$ with $M'\in\B_{\bA}^{\loc}$ (Lemma \ref{lemtoric}). 
 Thus, the ribbon structure on $\D$ of (b) 
coincides with the twist $\theta^{\mathrm s}$ that Proposition \ref{kostra} (iii) induces on $\B_{\bA}^{\loc}=\C_A^{\sloc}$. 
 Finally, $\theta^{\N}_{\tilde f}=\theta_\psi\otimes\id_{A'}=\varepsilon\,\id$ by Lemma 
\ref{leminduction} (v) (applicable 
since $\widetilde\theta_{A'}=\id$), and $\theta_\psi=\varepsilon$ in $\T^\varepsilon$ (both for $\T^-$, $\varepsilon=-1$, and for 
$\T^+$, $\varepsilon=+1$). Thus, $\D$ is $\varepsilon$-super-modular.
\end{proof}

\begin{remark}
\label{remenriched}
Let us discuss the enriched supercategory. 
Theorem \ref{thmBs} (1) concerns the underlying category $\D=\C_A^{\sloc}$ (even morphisms). Let 
$\Pi:=\tilde f\otimes_{A'}-$ be the parity shift. 
On $\D=\B_{\bA}^{\loc}$ this is $f\otimes-$.  
 Define the supercategory 
$\D^{\mathrm s}$ with the same objects and $\Hom^{\mathrm s}(M,N)=\Hom_\D(M,N)\oplus\Hom_\D(M,\Pi N)$ (even $\oplus$ odd), composition 
through $\Pi^2\cong\id$, induced by $\tilde f\otimes_{A'}\tilde f\cong\unit_\N$. 
 This is the correspondence of 
\cite{CKM24} between even morphisms into $\Pi N$ and odd morphisms into $N$. 
 We use it as a definition of the objects and morphisms of $\D^{\mathrm s}$.
  In $\D^{\mathrm s}$ the odd isomorphism $\unit\to\Pi\unit=\tilde f$ identifies $\tilde f$ 
with $\unit$. Call an object 
of $\D^{\mathrm s}$ \emph{transparent} if it is transparent in $\D$ 
(an object-level notion, involving only even morphisms).  
 We say that $\D^{\mathrm s}$ is \emph{non-degenerate in the super sense} if every transparent object is isomorphic 
in $\D^{\mathrm s}$ to $\unit^{\oplus n}$ for some $n$. 
 Then Theorem \ref{thmBs} (1) implies that $\D^{\mathrm s}$ is 
non-degenerate in the super sense. By (1) the transparent objects of $\D$ are the sums $\unit^a\oplus\tilde f^b$, and 
$\tilde f\cong\unit$ in $\D^{\mathrm s}$. In Section \ref{secVOSA} the identification of 
the even part is the assumption (v2). 
\end{remark}

\begin{remark}
\label{remcompletions} 
The functor $\iota$, and the proof of Theorem \ref{thmBs}, depend on the choice of 
$\T^\varepsilon$. Namely, $\N$ depends on it while $\D=\C_A^{\sloc}$ does not. 
 There are sixteen minimal modular extensions of $\sVect$ (Kitaev 16-fold way, 
\cite{BGHNPRW17}). In this paper we use the toric code.
  Lemmas \ref{lemtoric} and \ref{lemexact} use that $\T$ is 
pointed with $\langle\unit,\psi\rangle$ acting freely on the other simple objects. 
By Proposition \ref{propWitt}, $[\N]=[\C]$ in the Witt group of non-degenerate braided FTCs 
for the toric code. Corollary \ref{corext} is formulated in this case. 
 For nice VOSAs with $U=V_{\bar0}$, \cite{16fold} shows that $\Rep(V_{\bar0})$ 
is a modular category containing the NS category as the centralizer of a fermion. 
This is consistent with Example \ref{exIsing} where $\N$ is again an Ising-type category.
\end{remark}

\section{Constructions}

\subsection{Graded simple currents}

Let $H\le\Inv(\C)$ be a subgroup of the (finite) group of isomorphism classes of invertible objects of a braided FTC and 
$\chi:H\to\mathbb Z/2$ a character (the grading is part of the data). Let $X_h$ be an invertible object of class $h$, 
$q(h):=q(X_h)$, and $A_H=\bigoplus_hX_h$ with $\deg X_h=\chi(h)$. Then we have

\begin{proposition}
\label{propsimple-current} 

(a) Let us assume: $q(h)=(-1)^{\chi(h)}$ for all $h\in H$. Then $A_H$ (with the multiplication constructed in the 
proof) is a haploid supercommutative superalgebra, $\bA_H$ is exact (indeed separable), and $\C_{A_H}^{\sloc}$ is a 
braided FTC. 

(b) Assume in addition: $\C$ is ribbon and $\theta_{X_h}=\varepsilon^{\chi(h)}\id$ for all $h\in H$. Then $\bA_H$ is symmetric 
Frobenius and $\C_{A_H}^{\sloc}$ is ribbon. 

(c) With assumptions of (a), (b) and that $\C$ is an MTC, then $\C_{A_H}^{\sloc}$ is $\varepsilon$-super-modular.
\end{proposition}

\begin{proof} 
\emph{1): Trivial monodromy on $H$.} The monodromy bicharacter is 
$\mon_{g,h}=q(gh)q(g)^{-1}q(h)^{-1}$ (polarization of the quadratic form of the pointed braided subcategory, 
\cite{EGNO16}). By assumption of Proposition \ref{propsimple-current} (a) $q=(-1)^\chi$ is a character, thus $\mon_{g,h}=1$.

\emph{2): A braided model.} Let $Y_h\in\B$ be $X_h$ placed in degree $\chi(h)$. As $\chi$ is a character, 
$Y_gY_h\cong Y_{gh}$ and $h\mapsto Y_h$ is an injective homomorphism $H\to\Inv(\B)$. The braided subcategory 
$\langle Y_h\rangle\subset\B$ has quadratic form $q^{\mathrm s}(h)=c^{\mathrm s}_{Y_h,Y_h}=(-1)^{\chi(h)^2}q(h)=1$ 
 by assumption of Proposition \ref{propsimple-current} (a). By  
 Lemma \ref{factJS} there is a braided monoidal equivalence $G:\Vect_H\to\langle Y_h\rangle$ with $G(\delta_h)\cong Y_h$, where $\Vect_H$ has 
trivial associator and trivial braiding. 

\emph{3): The multiplication.} The group algebra $\kk[H]=\bigoplus_h\delta_h$ (multiplication $\delta_g\otimes\delta_h\to\delta_{gh}$) is a 
commutative algebra in $\Vect_H$, since the braiding is trivial. Set $\bA_H:=G(\kk[H])$. As $G$ is braided monoidal, 
$\bA_H$ is a commutative algebra in $\B$ with underlying object $\bigoplus_hY_h=(A_0,A_1)$, $A_i=\bigoplus_{\chi(h)=i}X_h$, and 
its unit and multiplication preserve the degree since $\chi(g)+\chi(h)=\chi(gh)$. By Proposition \ref{propbijection} (a), $A_H$ is a 
supercommutative superalgebra in $\C$. Explicitly, the multiplication consists of isomorphisms $m_{g,h}:X_g\otimes X_h\to X_{gh}$ (the components 
of the tensor structure of $G$). 

We separate three points: 
(i) \emph{Triviality of the associativity class.} That the $m_{g,h}$ satisfy the associativity axiom says that they trivialize the associator 
of the pointed category $\langle X_h\rangle\subset\C$, whose class belongs to $H^3_{\mathrm{ab}}(H,\kk^\times)\cong\mathrm{Quad}(H)$. This class is 
trivial since the assumption of Proposition \ref{propsimple-current} (a)  
 makes the quadratic form of $\langle X_h\rangle$ equal to $(-1)^\chi$, that of $\Vect_H$ with trivial associator 
and braiding $\beta(g,h)=(-1)^{\chi(g)\chi(h)}$ \cite{JS93}. 

(ii) \emph{Choice of the multiplication.} Triviality of the class does not specify the $m_{g,h}$: the tensor structures with 
the required property form a torsor under the normalized $2$-cocycles $Z^2(H,\kk^\times)$ (rescaling each $m_{g,h}$). 

(iii) \emph{Commutativity.} The condition $\mu c^{\mathrm s}=\mu$ in $\B$ 
makes the ratio of two admissible choices to be a symmetric   
$2$-cocycle. Since $\kk^\times$ is divisible, $\mathrm{Ext}^1_{\mathbb Z}(H,\kk^\times)=0$, thus every symmetric $2$-cocycle on the finite 
abelian group $H$ is a coboundary. Hence, any two commutative multiplications are isomorphic (by rescaling the identifications of $X_h$), 
and $A_H$ is well defined up to isomorphism. The unit and associativity axioms hold since $G$ is monoidal.

\emph{4): Haploid, separable, exact.} $\Hom_\B(\unit,\bA_H)=\bigoplus_h\Hom(\unit,Y_h)=\kk$, since $X_h\not\cong\unit$ for $h\neq1$ and 
$\chi(1)=0$. Next $\kk[H]$ is separable in $\Vect_H$: $\Delta(\delta_h)=|H|^{-1}\sum_{g\in H}\delta_g\otimes\delta_{g^{-1}h}$ is a bimodule map with 
$\mu\Delta=\id$ ($|H|$ is invertible as $\mathrm{char}\,\kk=0$). Hence $\bA_H=G(\kk[H])$ is separable in $\B$. A separable algebra $\bA$ in a finite tensor 
category is exact: the multiplication $M\otimes\bA\to M$ of any right module $M$ splits as a map of right modules, so for $P$ projective in $\B$ 
the module $P\otimes M$ is a direct summand of the free module $(P\otimes M)\otimes\bA$, which is projective in $\B_{\bA}$ since $P\otimes M$ is 
projective in $\B$ \cite{EGNO16} and the free functor is left adjoint to the exact forgetful functor. Thus, $\B_{\bA}$ is an exact 
$\B$-module category, i.e., $\bA_H$ is exact. By Theorem \ref{thmBs} (0), $\C_{A_H}^{\sloc}$ is a braided FTC. This proves (a).

\emph{5): Frobenius and ribbon.} $\kk[H]$ is a Frobenius algebra in $\Vect_H$ (comultiplication $\Delta$ above, counit $\delta_h\mapsto\delta_{h,1}$), thus  
$\bA_H=G(\kk[H])$ is a Frobenius algebra ($G$ is strong monoidal). Under assumption of Proposition \ref{propsimple-current} (b), $\theta^{\mathrm s}_{Y_h}=\varepsilon^{\chi(h)}\theta_{X_h}=\varepsilon^{2\chi(h)}=\id$, hence 
$\theta^{\mathrm s}_{\bA_H}=\id$ and, $\bA_H$ being commutative Frobenius, it is symmetric by Proposition \ref{mirda} (iii) (applied in the ribbon category $\B$). 
Moreover, the Frobenius isomorphism is an isomorphism of right $\bA_H$-modules $\bA_H\cong\bA_H^*$ \cite{SY26}, thus $\bA_H^*$ is the unit of 
$\B_{\bA_H}^{\loc}$ and belongs to the M\"uger center. Theorem \ref{thmAs} (c) (the rigidity assumption being automatic by Theorem \ref{thmBs} (0)) shows that 
$\C_{A_H}^{\sloc}$ is ribbon. This proves (b). 

(c) is Theorem \ref{thmBs} (2): $\bA_H$ is haploid, commutative, exact and symmetric Frobenius.

\emph{Necessity of assumption of Proposition \ref{propsimple-current} (a).} 
If $\bA_H$ is commutative and the multiplication maps $m_{h,h}$ 
(isomorphisms onto the summands $X_{h^2}$, the nonzero components of 
$\mu$ on $Y_h\otimes Y_h$) are nonzero, then $\mu c^{\mathrm s}=\mu$ on the
 simple object $Y_h\otimes Y_h$ leads to 
$c^{\mathrm s}_{Y_h,Y_h}=\id$, i.e., assumption of Proposition \ref{propsimple-current} (a).  
\end{proof}

\begin{remark}
\label{remSY-bosonic}

For $\chi=0$, assumption of Proposition \ref{propsimple-current} (a) reads $q\equiv1$, and Proposition \ref{propsimple-current} (a) recovers the simple current algebra construction 
of \cite{SY26} (their \S6.1). If moreover $\theta_{X_h}\neq1$ for some $h$, then assumption of Proposition \ref{propsimple-current} (b) fails (it requires $\theta_{X_h}=\varepsilon^0=1$), so that parts (b), (c) do 
not apply, although part (a) does: this is the situation of a $\tfrac12\mathbb Z$-graded vertex operator algebra extension (case (2) of Section \ref{secthree}), 
for which $\theta_A\neq\id$ and $\C_A^{\loc}$ has no ribbon structure induced from $\theta$ by Proposition \ref{kostra} (iii) (consistent with 
\cite{CKM24}, where $e^{2\pi iL_0}$ is not a module map). Thus, (a) holds in all three cases of Section \ref{secthree} under assumption of Proposition \ref{propsimple-current} (a), while (b), (c) hold in 
cases (1) and (3) only.
\end{remark}

\section{Examples}
\label{secexamples}

\begin{example} 
\label{exIsing}

We discuss here the free fermion case.  
Let $\C$ be the Ising MTC (the category of $L(\tfrac12,0)=V_{\bar0}$, simples 
$\unit,\psi,\sigma$), $\varepsilon=-1$, $H=\{\unit,\psi\}$, $\chi=\id$. Then 
$q(\psi)=c_{\psi,\psi}=-1=(-1)^{\chi(\psi)}$ and $\theta_\psi=-1=\varepsilon$, thus Proposition  
\ref{propsimple-current} (a)-(c) apply: $A=\unit\oplus\psi$ (the free Majorana fermion VOSA) 
is a haploid supercommutative symmetric Frobenius superalgebra. The Frobenius-Perron 
dimensions are $\FPdim\C=4$, $\FPdim\B=8$, $\FPdim\widetilde\B=16$, $\FPdim\bA=2$, 
$\FPdim\N=\FPdim\widetilde\B/\FPdim(A')^2=16/4=4$, $\FPdim\C_A^{\sloc}=\FPdim\N/2=2$ 
(the formula for $\N$ is \cite{DMNO13}, applicable since $A'=(\unit,\unit)\oplus(\psi,\psi_T)$ 
is a connected \'etale algebra in the non-degenerate fusion category $\widetilde\B=\mathrm{Ising}\boxtimes\T$). 
Explicitly, the objects $X\boxtimes t$ of $\mathrm{Ising}\boxtimes\T$ whose monodromy with $\psi\boxtimes\psi_T$ is trivial are 
$(\unit,\unit),(\psi,\unit),(\unit,\psi_T),(\psi,\psi_T)$ and $(\sigma,e),(\sigma,m)$ (use $\mon_{\sigma,\psi}=-1$, 
$\mon_{\psi,\psi}=1$, $\mon_{e,\psi_T}=\mon_{m,\psi_T}=-1$). Their free modules $F(-)$ are local (Lemma 
\ref{leminduction} (i)) and coincide pairwise: 
$F(\unit,\unit)\cong F(\psi,\psi_T)=\unit_\N$, $F(\psi,\unit)\cong F(\unit,\psi_T)=\tilde f$ (both of FP dimension $1$), 
$F(\sigma,e)\cong F(\sigma,m)=:s$. The module $s$ has underlying object $(\sigma,e)\oplus(\sigma,m)$ of FP dimension 
$2\sqrt2$, hence FP dimension $2\sqrt2/\FPdim A'=\sqrt2$ in $\N$. It is simple, since a non-zero object of $\N$ has FP dimension 
at least $1$, so $\sqrt2<2$ excludes a decomposition into two summands. As $\N$ is semisimple, 
$1^2+1^2+(\sqrt2)^2=4=\FPdim\N$ shows that $\unit_\N,\tilde f,s$ are all the simple objects of $\N$. Since $s$ is the only simple of 
dimension $\sqrt2$, it is self-dual, and $s\otimes s$ has FP dimension $2$ and contains $\unit_\N$ once, so 
$s\otimes s=\unit_\N\oplus\tilde f$: $\N$ has the fusion rules of the Ising category (we call such a braided fusion category Ising-type). 
The centralizer of $\tilde f$ is $\{\unit,\tilde f\}$ since $\mon_{s,\tilde f}=-1$. Thus, 
$\C_A^{\sloc}=\langle\tilde f\rangle\simeq\sVect$: the NS sector of the free fermion has one simple 
module up to parity. The Ramond sector arises from $\sigma$: $\mon_{\sigma,\psi}=-1$, thus $\sigma\otimes A$ is a 
graded $A$-module that is not super-local. Note that $\psi$ is not transparent in $\C$.
\end{example}

\begin{example}
\label{exSF}

Here we discuss the symplectic fermions, $d=1$. First, recall the following   

\begin{proposition}
\label{slon}

\textup{(1)} $\C=\Rep(W(2))$ ($W(2)$ the triplet algebra at $c=-2$, equal to the even part 
$\mathrm{SF}_1^+$ of the symplectic fermion VOSA $V=\mathrm{SF}_1$ \cite{Abe07}) is a 
non-degenerate braided FTC with ribbon twist $e^{2\pi iL_0}$;

\textup{(2)} the odd part $V_{\bar1}=W(1)$ (the irreducible module of lowest weight $1$) 
is invertible, $W(1)\boxtimes W(1)\cong W(0)=\unit$ \cite{TW13};

\textup{(3)} $V$ is a supercommutative superalgebra in $\C$ with $A_0=\unit$, $A_1=W(1)$, 
whose local modules are the NS $V$-supermodules, i.e., (v1),(v2) of Section \ref{secVOSA}, 
with non-zero product $V_{\bar1}\otimes V_{\bar1}\to V_{\bar0}$.
\end{proposition}

\emph{Categorical consequences of Proposition \ref{slon} (1)-(3).} $\theta_{W(1)}=e^{2\pi i\cdot1}=+1$, thus 
$\varepsilon=+1$ ($\mathbb Z$-graded VOSA; ``incorrect statistics'' in the sense of 
\cite{CKM24}). Since $\Hom(W(1)\otimes W(1),\unit)$ is one-dimensional, 
$c_{W(1),W(1)}=q\,\id$, and $\mu c=-\mu$ with $\mu\neq0$ leads to $q(W(1))=-1$. 
Then $\dim W(1)=\theta/q=-1$, matching the value of \cite{CKM24}. Thus, assumptions of Proposition \ref{propsimple-current} (a) and (b)  
hold for $H=\{\unit,W(1)\}$, $\chi=\id$, $\varepsilon=+1$, and Proposition  
\ref{propsimple-current} (c) shows that $\C_A^{\sloc}$ is a $(+1)$-super-modular category 
(in particular not pseudo-unitary, Remark \ref{remsphericity}). The dimension formula 
$\FPdim\N=\FPdim\widetilde\B/\FPdim(A')^2$ of \cite{LW23} (applicable since $A'$ is a special Frobenius 
algebra) gives, with $\FPdim\bA=2$, $\FPdim\N=4\FPdim\C/2^2=\FPdim\C$ and $\FPdim\C_A^{\sloc}=\FPdim\C/2$. 
By Proposition \ref{slon} (3) it is $\Rep^{NS}(\mathrm{SF}_1)$. 
Exactness and the Frobenius property of $\bA$ come from Proposition \ref{propsimple-current}.  
\end{example}

\section{The VOSA notions}
\label{secVOSA}

In this section we recall material from \cite{HKL15, CKL20, CKM24, HLZ}.  
 Let $V=V_{\bar0}\oplus V_{\bar1}$ be a VOSA and $U\subseteq V_{\bar0}$ a vertex subalgebra 
with the same conformal vector. $\C$ is a category of $U$-modules containing 
$V_{\bar0},V_{\bar1}$. Concerning a VOSA we state the following assumptions:

\textup{(v1)} \cite{CKL20, CKM24}: $V$ is a supercommutative superalgebra in $\C$ 
(Definition \ref{defsuperalgebra}) with $A_i=V_{\bar i}$. This is an algebra in the supercategory 
$\mathcal{SC}$ of \cite{CKM24} with even structure maps. The braiding of $\mathcal{SC}$ is 
$\bigoplus(-1)^{i_1i_2}\mathcal R_{W_1^{i_1},W_2^{i_2}}$, i.e., our $c^{\mathrm s}$.

\textup{(v2)} \cite{CKM24}: the category of NS (untwisted) $V$-supermodules in $\C$ with 
even morphisms is equivalent, as a braided monoidal category, to $\C_A^{\sloc}=\B_{\bA}^{\loc}$.

\textup{(v3)} Three separate facts on $\C$, which are not to be conflated: 

\textup{(v3a)} \emph{Tensor structure}: $\C$ carries the braided tensor structure of \cite{HLZ} (its existence is an assumption 
 on $U$, e.g., cofiniteness conditions as in \cite{Hua09}); 

\textup{(v3b)} \emph{Finiteness and rigidity}: $\C$ is a braided finite tensor category, i.e., finite abelian and rigid 
(for strongly rational $U$ see \cite{Hua08}; for strongly finite $U$ with rigid $\Rep(U)$ see \cite{Hua09} for the 
braided monoidal structure and \cite{CMSY25} for rigidity; otherwise it is an assumption); 

\textup{(v3c)} \emph{Twist}: the ribbon twist of $\C$ (when $\C$ is ribbon) is $\theta=e^{2\pi iL_0}$.

\textup{(v4)} $L_0$ acts semisimply on $V_{\bar1}$ (automatic for the VOSA $V$).

\textup{(v5)} All weights of $V_{\bar0}$ are integral, i.e., $\theta_{V_{\bar0}}=\id$ (used only for ribbon statements).

\subsection{The three cases}
\label{secthree}

\cite{CKM24} distinguishes: (1) $\mathbb Z$-graded VOSAs; (2) $\tfrac12\mathbb Z$-graded 
VOAs; (3) $\tfrac12\mathbb Z$-graded VOSAs whose odd part has half-integral weights 
(``correct statistics''). Assume $V=V_+\oplus V_-$ with $V_-$ a self-dual simple current, 
$V_-\boxtimes V_-\cong V_+$, and the product $V_-\otimes V_-\to V_+$ non-zero,  
 so that $\Hom(V_-\otimes V_-,V_+)$ is one-dimensional. In \cite{CKM24} this is 
derived when $V_+$ is regular. In the non-semisimple case it is a starting point 
(in Example \ref{exSF} it is Proposition \ref{slon} (2)). Then we have the following table: 
\begin{center}
\begin{tabular}{lcccccc}
Case & $\chi$ & $q(V_-)$ & $\theta_{V_-}$ & $\dim V_-$ & $\varepsilon$ & \cite{CKM24}: 
$\mathrm{qdim}\,V_-$\\\hline
(1) $\mathbb Z$-graded VOSA & $1$ & $-1$ & $+1$ & $-1$ & $+1$ & $-1$ incorrect\\
(2) $\tfrac12\mathbb Z$-graded VOA & $0$ & $+1$ & $-1$ & $-1$ & - & $-1$ incorrect\\
(3) $\tfrac12\mathbb Z$-graded VOSA & $1$ & $-1$ & $-1$ & $+1$ & $-1$ & $+1$ correct 
\end{tabular}
\end{center}

\subsection{Derivation of the table.} 

Let $\chi\in\{0,1\}$ be the parity of $V_-$ in the 
superalgebra $V$ ($\chi=1$ in the 
VOSA cases (1), (3), $\chi=0$ in the VOA case (2)).  

\emph{Twist.} $V_-$ is a simple current, thus $L_0$ acts on it with a single class of weights 
modulo $\mathbb Z$ (by (v4) it acts 
semisimply), and $\theta_{V_-}=e^{2\pi iL_0}=e^{2\pi i h}$ where $h$ is any weight of $V_-$: 
this is $+1$ if the weights are integral 
 (case (1)) and $-1$ if they are half-integral (cases (2), (3)). 
 In the two VOSA cases $\varepsilon=\theta_{V_-}$ by definition of $\varepsilon$ 
(see Definition \ref{defprops}). 

\emph{Quadratic form.} By (v1) the non-zero product $\mu:V_-\otimes V_-\to V_+$ satisfies 
$\mu c_{V_-,V_-}=(-1)^{\chi\cdot\chi}\mu=(-1)^\chi\mu$. 
 Since $c_{V_-,V_-}=q\,\id$ on the simple object $V_-\otimes V_-\cong V_+$, this gives $q=(-1)^\chi$. 

\emph{Dimension.} $\dim V_-=\theta_{V_-}/q$ by Remark \ref{remtheta-vs-c}. 
The last column contains the quantum dimension of $V_-$ as given in \cite{CKM24} in each case. The entries of the first columns are 
 given by (v1) and the weights. The agreement of the last column with $\theta/q$ is a consistency check of our sign 
conventions $c^{\mathrm s},\theta^{\mathrm s}$ against \cite{CKM24}. 

 In all three cases assumption of Proposition \ref{propsimple-current} (a) holds,  
 i.e., $q=(-1)^\chi$. Hence, Proposition \ref{propsimple-current} (a) applies to $H=\{\unit,V_-\}$, 
$\chi$. Assumption of Proposition \ref{propsimple-current} (b) holds in cases (1) and (3) 
(with $\varepsilon=\theta_{V_-}$). Thus, Proposition \ref{propsimple-current} (b), (c) apply there. 
In case (2) $\chi=0$ and $\theta_{V_-}=-1\neq\varepsilon^0$, thus assumption of Proposition
 \ref{propsimple-current} (b) fails: $\theta_A\neq\id$ and $\C_A^{\loc}$ has no ribbon 
structure induced from $\theta$ (consistent with \cite{CKM24}: $e^{2\pi iL_0}$ is not a 
module map there), see Remark \ref{remSY-bosonic}. 

\subsection{The twist}

 For $M=(M_0,M_1)\in\C_A^{\sloc}$ (an NS $V$-supermodule) the grading is the $\mathbb Z/2$-grading of 
the supermodule, $P_M$ is the parity involution ($+1$ on $M_0$, $-1$ on $M_1$), and the vertex operators satisfy 
$Y_M(v,x)M_i\subset M_{i+|v|}((x))$, so $P_MY_M(v,x)P_M^{-1}=(-1)^{|v|}Y_M(v,x)$. In the NS sector the modes $v_n$ of $Y_M(v,x)$ are 
indexed by $n\in\mathbb Z$, so that for $v$ of weight $\mathrm{wt}\,v$ one has $e^{2\pi iL_0}v_ne^{-2\pi iL_0}=e^{2\pi i\,\mathrm{wt}\,v}v_n$.

On $M=(M_0,M_1)\in\C_A^{\sloc}$, 
$\theta^{\mathrm s}_M=\theta_{M_0}\oplus\varepsilon\theta_{M_1}$. For $\varepsilon=+1$ 
this is $e^{2\pi iL_0}$. For $\varepsilon=-1$ it 
is $e^{2\pi iL_0}P_M$. In case (3) it is a $V$-module endomorphism: conjugation by $e^{2\pi iL_0}$ 
multiplies $Y(v,x)$, $v$ odd, by $e^{2\pi i\,\mathrm{wt}\,v}=-1$, as does conjugation by $P$,
thus the product commutes with $Y(v,x)$. On $V$, $e^{2\pi iL_0}P_V=\id$ (\cite{CKM24}: 
$\theta_V=P_V$ for correct statistics), as required for a ribbon structure on a category
 with unit $V$. Thus, the answer to the question ``does $\theta^{\mathrm s}$ agree with the twist 
$e^{2\pi iL_0}$ on NS supermodules?'' is:  for $\varepsilon=+1$, and up to the parity 
operator for $\varepsilon=-1$, where the correction is forced.

\subsection{Applying the theorems}

\begin{corollary}
\label{corVOSA}
Assume \textup{(v1)-(v4)} (and \textup{(v5)} for part \textup{(c)}) and that $\bA$ 
is exact and haploid. Then: 
\begin{enumerate}
\item[(a)] {\rm Braided conclusion:} $\Rep^{NS}(V)$ is a rigid braided finite tensor category.

\item[(b)] {\rm Super-non-degeneracy:} if moreover $\C$ is non-degenerate, then $\Rep^{NS}(V)$ is 
super-non-degenerate and is the centralizer of a fermion in a non-degenerate braided FTC $\N$.

\item[(c)] {\rm Ribbon conclusion:} if in addition $\C$ is ribbon with twist 
$\theta=e^{2\pi iL_0}$ ({\rm v3c}), 
$\theta_{V_{\bar1}}=\varepsilon\,\id$ for a sign $\varepsilon$ (all weights of $V_{\bar1}$ integral, 
resp.\ half-integral), $\C$ is non-degenerate and $\bA$ is Frobenius, then $\bA$ is symmetric 
Frobenius, $\N$ is an MTC, and $\Rep^{NS}(V)$ is ribbon with the twist $\theta^{\mathrm s}$ and $\varepsilon$-super-modular.
\end{enumerate}
\end{corollary}

\begin{proof}
(a) is Theorem \ref{thmBs} (0) with (v2). (b) is Theorem \ref{thmBs} (1). 
(c): by (v5) and the assumption on $V_{\bar1}$, 
$\theta_{A_0}=\id$ and $\theta_{A_1}=\varepsilon\,\id$, i.e., $\theta^{\mathrm s}_{\bA}=\id$ 
(Proposition \ref{propbijection} (c)).  
A Frobenius commutative algebra with $\theta^{\mathrm s}_{\bA}=\id$ is symmetric by 
Proposition \ref{mirda} (iii), and Theorem 
\ref{thmBs} (2) applies.
\end{proof}

\begin{lemma}
\label{lemsimple-exact}
Assume \textup{(v1)-(v3)}, $\Hom_\C(\unit,V_{\bar0})=\kk$, and that $V$ is simple as a 
$V$-supermodule (no proper nonzero graded submodule). Then $\bA$ is a simple algebra in $\B$, 
hence exact.
\end{lemma}

\begin{proof}
We use the definition of \cite{CMSY25} (also \cite{SY26}): $\bA$ is simple if its only two-sided ideals 
(sub-bimodules of $\bA$) are $0$ and $\bA$. Let $I\subset\bA$ be \emph{any} right submodule, with inclusion $j$, 
$ja_I=\mu(j\otimes\id)$. We show that $I$ is local: with $\mon_I=c_{A,I}c_{I,A}$, naturality of $c$ gives 
$j\,a_I\,\mon_I=\mu(j\otimes\id)\mon_I=\mu\,\mon_{\bA,A}\,(j\otimes\id)=\mu(j\otimes\id)
=j\,a_I$ ($\bA$ is local over itself, by commutativity: $\mu c^2=\mu$), and $j$ is a 
monomorphism, thus $a_I\mon_I=a_I$. By (v2) $I$ is a graded $V$-submodule of $V$, 
thus $I\in\{0,V\}$. Hence, the only right submodules of $\bA$ are $0$, $\bA$. 
 This is stronger than what is needed, 
since every two-sided ideal is in particular a right submodule, 
thus the only two-sided ideals are $0$, $\bA$ and $\bA$ 
is simple. By Lemma \ref{factsimple} \cite{CSZ25}, valid for algebras in finite tensor 
categories. Here $\B$ is an FTC by (v3b)), $\bA$ is exact. 
\end{proof}

The next corollary concerns the fusion case. 

\begin{corollary}
\label{corfusion}

Assume \textup{(v1)-(v3)}, $\C$ a braided fusion category (e.g., $U$ strongly rational), 
$V$ simple, $\Hom_\C(\unit,V_{\bar0})=\kk$. Then $\C_{A}^{\sloc}=\B_{\bA}^{\loc}$ is a 
braided fusion category.
\end{corollary}

\begin{proof}
$\B=\C\boxtimes\sVect$ is fusion and $\bA$ is a simple commutative algebra in it 
(first part of Lemma \ref{lemsimple-exact}), thus \cite{CMSY25} Lemma \ref{factfusion} applies.  
\end{proof}

This is one categorical ingredient (for the even part $V_{\bar0}$ strongly rational) of the strong-rationality statement for 
extensions announced in \cite{CMSY25}; we do not claim the strong-rationality statement. 

\section{Witt classes}
\label{secWitt}

Let $\Wns$ be the Witt group of non-degenerate braided FTCs, with Witt equivalence as in Lemma 
\ref{factWitt}: $\C\sim\D$ iff $\C\boxtimes\Z(\mathcal A)\simeq\D\boxtimes\Z(\mathcal B)$ for some FTCs $\mathcal A,\mathcal B$. 
We write $\Wns$ additively. We use only the following consequence of \cite{SY26} (see Lemma \ref{factWitt}): if $\C$ is 
non-degenerate and $A\in\C$ is a haploid commutative exact algebra, then $\C_A^{\loc}$ is non-degenerate and Witt 
equivalent to $\C$. 

\begin{proposition}
\label{propWitt}
Let $\C$ be a non-degenerate braided FTC and $A$ a haploid supercommutative superalgebra with $\bA$ exact. Then 
$\N=\widetilde\B_{\iota(\bA)}^{\loc}$ is non-degenerate and $[\C]=[\N]$ in $\Wns$.
\end{proposition}

\begin{proof}
$\widetilde\B=\C\boxtimes\T^\varepsilon$ is non-degenerate (Section \ref{secproofB}) and $\iota(\bA)$ is haploid, commutative 
and exact by Lemma \ref{lemexact}. Thus, $[\widetilde\B]=[\N]$ by Lemma \ref{factWitt}. Next, $\T=\Z(\Vect_{\mathbb Z/2})$, thus 
$\T\boxtimes\Z(\Vect)\simeq\Vect\boxtimes\Z(\Vect_{\mathbb Z/2})$ and $[\T]=[\Vect]=0$, hence 
$[\widetilde\B]=[\C]+[\T]=[\C]$. The choice of $\varepsilon$ is irrelevant: $\T^+$ and $\T^-$ have the same underlying braided 
category (Section \ref{sectoric}), and $\Wns$ is defined for braided categories.
\end{proof}

We do not define super Witt group here: it would require the fermionic product 
$(\D_1\boxtimes\D_2)^{\loc}_{\unit\oplus f_1\boxtimes f_2}$ (condensation of the boson $f_1\boxtimes f_2$), a proof 
that it is well defined on classes, and non-semisimple analogues of \cite{DMNO13}.

\section{Comparison with existing formulations}
\label{seccomparison}

The remaining open questions are the super Witt group, other modular completions, the Ramond sector, 
odd morphisms for modules beyond Remark \ref{remenriched}, 
and the material of Section 4.3 of \cite{CMSY25}.   
   The material of this paper is also useful in computations of topological invariants
 in non-commutative geometry setup in 
 condensed matter theory 
\cite{chernodub2017scale, zhang2020influence, zubkov2012momentum, zubkov2017topology, kmmzz}, 
and computations of characteristic classes in foliation theory \cite{BG, BGG, BEG, Z1, Z2}. 
Next we compare our formulation with those of \cite{CKM24} and \cite{CMSY25}.

\subsection{Creutzig-Kanade-McRae} 

The comparison concerns the introduction and Sections 2.1-2.2 of \cite{CKM24}.

\emph{The category.} \cite{CKM24} defines $\mathcal{SC}$ with objects $W=(W^{\bar0},W^{\bar1})$ and 
$\Hom_{\mathcal{SC}}(W_1,W_2)=\Hom_\C(W_1^{\bar0}\oplus W_1^{\bar1},W_2^{\bar0}\oplus W_2^{\bar1})$, graded into even and odd morphisms. 
Its underlying category has the even morphisms \cite{CKM24}, and an odd morphism $W_1\to W_2$ is an even morphism 
$W_1\to\Pi W_2$. The tensor product is $(W_1\boxtimes W_2)^i=\bigoplus_{i_1+i_2=i}W_1^{i_1}\boxtimes W_2^{i_2}$ and the braiding 
is $\bigoplus(-1)^{i_1i_2}\mathcal R_{W_1^{i_1},W_2^{i_2}}$ \cite{CKM24}. 
These are the definitions of Section \ref{secgraded}: 
$\B$ is the underlying category of $\mathcal{SC}$, with $\Pi=f\otimes-$.

\emph{Algebras and modules.} A superalgebra in $\C$ is an algebra in $\mathcal{SC}$ with even 
structure maps, and $\mathrm{Rep}^0A$ is the category of local modules with even structure
 maps \cite{CKM24}. Its underlying category is $\B_{\bA}^{\loc}$ by Proposition 
\ref{propbijection}.

\emph{Twist.} \cite{CKM24} states that $\theta_V=e^{2\pi iL_0}$ and the parity involution $P_V$ coincide for correct 
statistics and differ for incorrect statistics. In our notation these are the cases $\varepsilon=-1$ and $\varepsilon=+1$ of the 
single condition $\theta_{A_1}=\varepsilon\,\id$ (see Section \ref{secVOSA}).
  Note that the braiding signs agree. The dimensions $\mathrm{qdim}\,V_-$ in the three cases 
agree with $\theta/q$ computed from 
our data (see Section \ref{secthree}). NS modules are local and Ramond modules are twisted 
(non-local), as in Example \ref{exIsing}.

\emph{Odd morphisms.} The notion ``odd morphism equivalent to even morphism into $\Pi$'' is 
stated in \cite{CKM24} for $\mathcal{SC}$. 
For modules we use it as the definition of $\D^{\mathrm s}$ (see Remark \ref{remenriched}).
 Our Theorems \ref{thmBs} (0)-(2) concern possibly non-semisimple $\C$ and 
exact (not necessarily separable) $\bA$, and give the M\"uger-center statement 
for the underlying category of $\mathrm{Rep}^0A$. 

\subsection{Creutzig-McRae-Shimizu-Yadav} 

The comparison concerns the introduction and Sections 1-3 of \cite{CMSY25} 
(through Section 3.5), and the two statements Theorem 2.23 
and Theorem 2.26 used in Section \ref{secexternal}. Section 4.3 is mentioned below 
only through the introduction.

\emph{The correspondence between left and right modules.} \cite{CMSY25} use left modules $(M,\mu_M:A\otimes M\to M)$, 
with $\C_A^{\loc}=\{\mu_M\,c_{M,A}\,c_{A,M}=\mu_M\}$, and $\otimes_A$ defined through $c^{-1}_{A,-}$. Let $R$ send a right 
module $(M,a_M)$ to the left module $(M,\lambda_M)$ with $\lambda_M=a_Mc_{A,M}$ (the left action of Section \ref{secinduction}), 
and be the identity on morphisms. $R$ is an isomorphism of categories from right to left $A$-modules, with inverse 
$a_M=\lambda_Mc_{A,M}^{-1}$. Substituting $a_M=\lambda_Mc_{A,M}^{-1}$ one finds that   
$a_M\,c_{A,M}\,c_{M,A}=a_M$ iff $\lambda_M\,c_{M,A}=\lambda_M\,c_{A,M}^{-1}$ 
which is equivalent to  $\lambda_M\,c_{M,A}\,c_{A,M}=\lambda_M$,   
thus $R$ restricts to an isomorphism between our category of local right modules and the 
category of local left modules 
of \cite{CMSY25}. For local $N$, $\lambda_N=a_Nc_{A,N}=a_Nc_{N,A}^{-1}$ 
(Section \ref{secinduction}). Thus, the left action 
associated with a local module does not depend on whether $c$ or $c^{-1}$ is used to pass between right and left 
actions, the tensor products $\otimes_A$ defined with $c$ and with $c^{-1}$ coincide on local modules, and we use in this paper the braiding induced by $c$, as in \cite{Par95}.

\emph{Results used:} for a braided FTC $\C$ and $A$ exact commutative haploid, $\C_A$ is a finite tensor category and $\C_A^{\loc}$ 
a braided FTC, non-degenerate if $\C$ is \cite{CMSY25}. Applied to $(\B,\bA)$ this is 
our Theorem \ref{thmBs} (0) and the 
finite-tensor-category part of (1). For a simple commutative algebra in a fusion category, \cite{CMSY25} gives 
Corollary \ref{corfusion}.

\emph{The conditions $\theta_A^2=\id$ and $\theta_A=\id$.} \cite{SY26} show that $D(A)=A^*$ is local iff 
$\theta_A^2=\id$ \cite{SY26}. This is the condition under which, according to the introduction of \cite{CMSY25}, 
$\C_A^{\loc}$ is a braided Grothendieck-Verdier category with dualizing object $D(A)$. For a ribbon structure one needs the 
stronger condition $\theta_A=\id$ together with $D(A)\in(\C_A^{\loc})'$ \cite{SY26} (Proposition \ref{kostra} (iii)). In our 
setting the relevant twist is $\theta^{\mathrm s}$ of $\B$, and the conditions are $(\theta^{\mathrm s}_{\bA})^2=\id$, i.e.,  
$\theta_{A_0}^2=\id$ \emph{and} $\theta_{A_1}^2=\id$ (automatic if $\theta_{A_i}=\pm1$), resp.\ $\theta^{\mathrm s}_{\bA}=\id$ 
(Proposition \ref{propbijection} (c)).

\emph{Section 4.3.} The introduction of \cite{CMSY25} states that Section 4.3 treats extensions $V\subseteq A$ with $V$ a 
$\mathbb Z$-graded VOA and $A$ a $\mathbb Z$- or $\tfrac12\mathbb Z$-graded VOSA containing $V$ in its even part, and that 
Theorem 4.21 is a VOSA generalization of Theorem 1.1 there (a simple $\mathbb Z_{\ge0}$-graded extension of a strongly 
rational VOA is strongly rational). Corollary \ref{corfusion} gives one corresponding 
categorical ingredient in our formalism. 

\section*{Acknowledgments}
The author is supported by the Institute of Mathematics, Academy of Sciences of the Czech Republic (RVO 67985840). 
The author would like to thank S. Kanade and P. Zusmanovich for valuable discussions.

\textbf{Data Availability.} Data sharing is not applicable to this article as no datasets were generated or 
analysed during the current study.

\textbf{Declarations}
This paper arose in partial collaboration with ChatGPT, an AI assistant developed by OpenAI. 
The collaboration involved exploring proof strategies, refining arguments, locating references, 
and drafting. Revision and writing of final text have been done by the author. The author takes 
full responsibility for the mathematical content and the final text.

\textbf{Conflict of interest.} The author has no conflicts of interest to declare that are relevant 
to the content of this article.


\end{document}